\documentclass[11pt]{article}
\usepackage[x11names]{xcolor}
\usepackage{graphicx}
\usepackage{enumerate}
\usepackage{amsthm,amssymb,amsmath,mathrsfs}
\usepackage{enumitem,mathtools}
\usepackage[ruled,vlined,linesnumbered]{algorithm2e}
\usepackage{url}
\usepackage{fullpage}
\usepackage{array}
\usepackage{multirow}
\usepackage{longtable}
\usepackage{tikz}
\usepackage{bbm}
\usepackage{bm}
\usepackage{physics}
\usepackage{scalerel,stackengine}
\usepackage{hyperref}
\usepackage{hypcap}

\hypersetup
{
    colorlinks=true,
    linkcolor=RoyalBlue3,
    citecolor=SpringGreen3,
    urlcolor=Turquoise3
}

\SetAlgoCaptionLayout{centerline}

\makeatletter
\newcommand{\myfnsymbol}[1]{%
  \expandafter\@myfnsymbol\csname c@#1\endcsname
}
\newcommand{\@myfnsymbol}[1]{%
  \ifcase #1
  \or 
  \or 1
  \or $*$
  \or $\dagger$
  \or $\ddagger$
  \fi
}
\newcommand{\spmsntu}{\@myfnsymbol{2}}
\newcommand{\emailA}{\@myfnsymbol{3}}
\newcommand{\emailB}{\@myfnsymbol{4}}
\newcommand{\emailC}{\@myfnsymbol{5}}
\makeatother

\makeatletter
\newcommand*{\transpose}{%
  {\mathpalette\@transpose{}}%
}
\newcommand*{\@transpose}[2]{%
  \raisebox{\depth}{$\m@th#1\intercal$}%
}

\makeatletter
\newcommand{\labeltext}[2]{%
  \@bsphack
  \MakeLinkTarget*{#1}%
  \def\@currentlabel{#1}{\label{#2}}%
  \@esphack
}
\makeatother

\makeatletter
\newcommand{\nosemic}{\renewcommand{\@endalgocfline}{\relax}}
\let\oldnl\nl
\newcommand{\nonl}{\renewcommand{\nl}{\let\nl\oldnl}}
\makeatother

\newtheorem{thm}{Theorem}[section]
\newtheorem{dfn}[thm]{Definition}

\newtheorem{lem}[thm]{Lemma}
\newtheorem{prop}[thm]{Proposition}
\newtheorem{remark}[thm]{Remark}
\newtheorem{ex}[thm]{Example}
\newtheorem{qn}[thm]{Question}

\DeclareMathOperator{\Char}{Char}
\DeclareMathOperator{\Min}{Min}
\DeclareMathOperator{\Quo}{Quo}
\DeclareMathOperator{\Disc}{Disc}
\DeclareMathOperator{\Var}{Var}
\DeclareMathOperator{\sgn}{sgn}

\DeclareMathOperator{\coeff}{coeff}
\DeclareMathOperator{\mult}{mult}
\DeclareMathOperator{\rem}{rem}

\usepackage{array,etoolbox}
\preto\tabular{\setcounter{magicrownumbers}{0}}
\newcounter{magicrownumbers}

\preto\tabular{\setcounter{magicrownumbers2}{0}}
\newcounter{magicrownumbers2}

\SetAlFnt{\footnotesize}

\title{Real-rooted integer polynomial enumeration algorithms and interlacing polynomials via linear programming}

\author{Gary R.W. Greaves \thanks{GRWG was partially supported by the Singapore Ministry of Education Academic Research Fund; grant numbers: RG14/24 (Tier 1) and MOE-T2EP20222-0005 (Tier 2).} \and Jeven Syatriadi}
\date{}

\begin{document}

\renewcommand{\thefootnote}{\myfnsymbol{footnote}}

\maketitle

\footnotetext[1]{Division of Mathematical Sciences, School of Physical and Mathematical Sciences, Nanyang Technological University, 21 Nanyang Link, Singapore 637371.}%
\footnotetext[1]{\textit{Email addresses}: \texttt{gary@ntu.edu.sg} (G.R.W. Greaves), \texttt{jeve0002@e.ntu.edu.sg} (J. Syatriadi).}%

\setcounter{footnote}{0}
\renewcommand{\thefootnote}{\arabic{footnote}}

\begin{abstract}
    We extend the algorithms of Robinson, Smyth, and McKee--Smyth to enumerate all real-rooted integer polynomials of a fixed degree, where the first few (at least three) leading coefficients are specified.
    Additionally, we introduce new linear programming algorithms to enumerate all feasible interlacing polynomials of a given polynomial that comes from a certain family of real-rooted integer polynomials.
    These algorithms are further specialised for the study of real equiangular lines, incorporating additional number-theoretic constraints to restrict the enumeration. 
    Our improvements significantly enhance the efficiency of the methods presented in previous work by the authors.
\end{abstract}

\section{Introduction} \label{sec:intro}

\subsection{Background} \label{subsec:background}

There is a long history and interest in the study of the properties of \emph{totally-real} and \emph{totally-positive algebraic integers}.
An important problem in this area of research, famously coined as the \emph{Schur--Siegel--Smyth trace problem} \cite{Borwein02}, dates back to a 1918 paper of Schur~\cite{Schur18}, along with the pioneering works of Siegel~\cite{Siegel45} and Smyth~\cite{Smyth84}.
The trace problem is covered in \cite[Section 14.3]{MS21} and a detailed chronology can be found in \cite{AP08}, which also contains Serre's letter to Smyth \cite[Appendix B]{AP08}.
More recently, Smith~\cite{Smith24} has made a significant advance on the trace problem, refuting a longstanding consensus.

In the 1960s, Robinson~\cite{Robinson64} devised and implemented a search algorithm to enumerate totally real algebraic integers of degree at most $6$ with span less than $4$.
Smyth~\cite{Smyth84} later adapted and improved Robinson's algorithm in relation to what is now known as the Schur--Siegel--Smyth trace problem. 
In 2004, McKee and Smyth~\cite{MS04} utilised the discriminant of a polynomial to further enhance the algorithm's efficiency and applied it to the study of special classes of algebraic integers.
Building upon this work, we develop algorithms to enumerate real-rooted integer polynomials, where we also include the \emph{reducible} ones, and particularly those that have multiple roots.

Part of our motivation stems from the study of \textit{real equiangular lines}.
A set of lines through the origin in an Euclidean space is called \emph{equiangular} if the angle between any pair of lines is the same \cite{GKMS16, LS73}.
Over the past few years, in a series of papers~\cite{GS24, GSY21, GSY23} the authors and Yatsyna have determined the maximum cardinality of systems of {equiangular lines} in various low dimensions.
A \emph{Seidel matrix} is a symmetric matrix whose diagonal entries are equal to zero and whose off-diagonal entries belong to the set $\{\pm 1\}$.
Due to the correspondence~\cite{GKMS16, LS73} between a Seidel matrix and an equiangular line system, the nonexistence of certain large systems of equiangular lines is equivalent to the nonexistence of certain Seidel matrices with prescribed characteristic polynomials.
An essential tool in our proofs was the enumeration of the characteristic polynomials of hypothetical Seidel matrices, which correspond to putative large systems of equiangular lines.
Our enumeration algorithms are based on modifications to the McKee--Smyth algorithm \cite{MS04}.
In this paper, we provide further improvements to these algorithms.
In addition to improving the efficiency of verifying our proofs for upper bounds on the cardinality of an equiangular line system, we anticipate that our algorithms are of independent interest.

Sets of interlacing polynomials have been the focus of recent mathematical breakthroughs, such as the construction of infinite families of bipartite Ramanujan graphs \cite{MSS15} and the solution of the Kadison--Singer problem \cite{MSS15II}.
Interlacing methods were also instrumental in the celebrated proof of the sensitivity conjecture \cite{Huang19}.
Real-rooted polynomials \cite{BHL88, CS07} and the theory of rational convex polytopes \cite{BZ06} have also found applications in combinatorics and graph theory.

\subsection{Main contributions and organisation} \label{subsec:main}

In this paper, we present all of our algorithms in the form of pseudocode.
Theorem~\ref{thm:realrootedalgo} in Section~\ref{sec:realrootedalgo} provides a rigorous assurance that our algorithms complete the search without missing any polynomials.
Furthermore, it allows us to implement \hyperref[alg:numrealrooted]{$\mathsf{AllRealRooted}$}, which extends the scope of the McKee--Smyth algorithm~\cite{MS04} to also cover the reducible polynomials.
To guarantee its termination, \hyperref[alg:numrealrooted]{$\mathsf{AllRealRooted}$} requires that at least the first three leading coefficients are specified.
The algorithm \hyperref[alg:endpoints]{$\mathsf{EndPoints}$} uses Lemma~\ref{lem:omegafloorceil} from Section~\ref{subsec:approxerr} to circumvent potential numerical approximation errors when determining integer interval endpoints required for \hyperref[alg:numrealrooted]{$\mathsf{AllRealRooted}$}.
We also use \hyperref[alg:endpoints]{$\mathsf{EndPoints}$} and Lemma~\ref{lem:omegafloorceil} for other algorithms in Section~\ref{sec:Seidelcharalgo} and Section~\ref{sec:LPalgo}.

In Section~\ref{sec:Seidelcharalgo}, we modify \hyperref[alg:numrealrooted]{$\mathsf{AllRealRooted}$} to enumerate \emph{Seidel-feasible polynomials}: monic real-rooted integer polynomials satisfying various necessary conditions that the characteristic polynomial of a Seidel matrix must satisfy.
The resulting algorithms are the main algorithms that we employ in \cite{GS24, GSY21, GSY23}.
We make some improvements by implementing additional number-theoretic constraints from \cite{GY19} in \hyperref[alg:checkmod]{$\mathsf{ModCheck}$}.
The algorithm \hyperref[alg:checkmod]{$\mathsf{ModCheck}$} refines and filters the output of \hyperref[alg:Seidelcharalgodd]{$\mathsf{FeasiblePartial}$}, which consists of polynomials of odd degree.
For \hyperref[alg:Seidelcharalgoeven]{$\mathsf{FeasibleEven}$}, the number-theoretic constraint is the relatively compact \emph{type~2} condition, which we will define in Section~\ref{sec:Seidelcharalgo}.

In Section~\ref{sec:LPalgo}, we derive Theorem~\ref{thm:gamma}, which enables us to leverage linear programming to enumerate the set of \emph{Seidel interlacing polynomials}.
Linear and semidefinite programming techniques have been applied to the problem of determining the maximum possible cardinality of equiangular lines \cite{BY14, LMFV22, KT19, Yu17}.
Our new algorithms \hyperref[alg:interlacingLPoddeven]{$\mathsf{InterlacingEven}$} and \hyperref[alg:interlacingLPevenodd]{$\mathsf{InterlacingPartial}$} substantially speed up the enumeration of the set of interlacing characteristic polynomials in \cite{GS24, GSY21, GSY23}.
For comparison, we evaluate the performance of our new algorithms on the computational tasks documented in \cite[Figure 5]{GSY23}.
There we can see that the overall computational time required for polynomials where their minimal polynomials are of degree $10$ is around 3 hours.
With a basic, direct implementation of \hyperref[alg:interlacingLPoddeven]{$\mathsf{InterlacingEven}$}, we can now complete the entire series of computations in under 5 minutes, achieving a speedup of roughly 85 to 90 times.

\section{Enumeration of real-rooted integer polynomials}
\label{sec:realrootedalgo}

\subsection{Approximation error in numerical analysis} \label{subsec:approxerr}

We briefly discuss approximation error in numerical analysis.
Let $\omega \in \mathbb{R}$ and let $\tilde{\omega}$ be a numerical approximation of $\omega$.
Then $\left\lvert \omega-\tilde{\omega} \right\rvert$ is referred to as the \textbf{absolute error} \cite{Gautschi12, Neumaier01, Ueberhuber97}.
In this paper, our primary concern regarding numerical errors is on determining the correct integer interval endpoints.
We will use Lemma~\ref{lem:omegafloorceil} below to avoid numerical approximation errors in our algorithms.

Let $x$ be a real number.
Following \cite{FS09}, we define $\lceil x \rfloor \coloneqq \left\lfloor x+\frac{1}{2} \right\rfloor$.
Here $\lceil x \rfloor$ is the nearest integer to $x$ and if $x = n+\frac{1}{2}$ for some integer $n$, we have that $\lceil x \rfloor = n+1$.

\begin{ex} \label{ex:approxerr}
    Let $\omega$ and $\tilde{\omega}$ be real numbers such that $\omega < 7 < \tilde{\omega}$ and $\left\lvert \omega - \tilde{\omega} \right\rvert < \varepsilon$ for some $\varepsilon <1/2$.
    Observe that $\lfloor \omega \rfloor = 6 = \left\lceil \tilde{\omega} \right\rfloor - 1$ and $\lceil \omega \rceil = 7 = \left\lceil \tilde{\omega} \right\rfloor$.
\end{ex}

Example~\ref{ex:approxerr} illustrates an instance where the exact value and its numerical approximation lie on the opposite sides of an integer.
The next lemma states a basic expression for the floor and ceiling of the exact value in terms of its numerical approximation, provided that the absolute error is less than $1/2$.

\begin{lem} \label{lem:omegafloorceil}
    Let $\omega$ and $\tilde{\omega}$ be real numbers such that $\left\lvert \omega - \tilde{\omega} \right\rvert < 1/2$.
    Then $\lceil \omega \rceil - \left\lceil \tilde{\omega} \right\rfloor \in \{0,1\}$ and $\lfloor \omega \rfloor - \left\lceil \tilde{\omega} \right\rfloor \in \{-1,0\}$.
\end{lem}

\begin{proof}
    Clearly, we have $\left\lvert \tilde{\omega} - \left\lceil \tilde{\omega} \right\rfloor \right\rvert \leqslant 1/2$.
    By the triangle inequality,
    \[
    \left\lvert \omega - \left\lceil \tilde{\omega} \right\rfloor \right\rvert \leqslant \left\lvert \omega - \tilde{\omega} \right\rvert + \left\lvert \tilde{\omega} - \left\lceil \tilde{\omega} \right\rfloor \right\rvert < 1 \implies \left\lceil \tilde{\omega} \right\rfloor - 1 < \omega < \left\lceil \tilde{\omega} \right\rfloor + 1.
    \]
    Hence, we conclude that $\lceil \omega \rceil = \left\lceil \tilde{\omega} \right\rfloor$ or $\lceil \omega \rceil = \left\lceil \tilde{\omega} \right\rfloor + 1$, while $\lfloor \omega \rfloor = \left\lceil \tilde{\omega} \right\rfloor - 1$ or $\lfloor \omega \rfloor = \left\lceil \tilde{\omega} \right\rfloor$.
\end{proof}

We will utilise Lemma~\ref{lem:omegafloorceil} repeatedly, throughout.

\begin{remark} \label{remark:round}
    Let $\mathsf{Round}(x) : \mathbb{R} \to \mathbb{Z}$ be a rounding-to-the-nearest-integer function with no particular rule of rounding choice when $x=n+\frac{1}{2}$ for some integer $n$.
    Since $\left\lvert \tilde{\omega} - \mathsf{Round}\left( \tilde{\omega} \right) \right\rvert \leqslant 1/2$, observe that Lemma~\ref{lem:omegafloorceil} still holds even if we replace the function $\lceil x \rfloor$ by $\mathsf{Round}(x)$.
\end{remark}

\subsection{Real-rooted integer polynomials with specified first few leading coefficients}
\label{subsec:realrootedtop3}

Let $p(x) \in \mathbb{R}[x]$ be a nonzero polynomial.
We call $p(x)$ \textbf{real-rooted} if for all $z \in \mathbb{C}$, $p(z) = 0$ implies that $z \in \mathbb{R}$.
An \textbf{integer polynomial} is an element of $\mathbb{Z}[x]$.
Let $p(x)$ be a nonzero integer polynomial.
We call $p(x)$ \textbf{reducible} if $p(x) = q(x) \cdot r(x)$ for some $q(x), r(x) \in \mathbb{Z}[x]$ where $q(x) \neq \pm 1$ and $r(x) \neq \pm 1$ \cite{Artin91, DF04}.
We present Algorithm~\ref{alg:numrealrooted}, which we refer to as \hyperref[alg:numrealrooted]{$\mathsf{AllRealRooted}$}, for enumerating real-rooted integer polynomials with fixed first few (at least three) leading coefficients and fixed degree.
All such polynomials that are reducible, including the ones that have multiple roots, are also enumerated by \hyperref[alg:numrealrooted]{$\mathsf{AllRealRooted}$}, which results from modifying an algorithm of McKee--Smyth~\cite{MS04}.
\hyperref[alg:numrealrooted]{$\mathsf{AllRealRooted}$} incorporates Lemma~\ref{lem:omegafloorceil}, which we implement in \hyperref[alg:endpoints]{$\mathsf{EndPoints}$}, to avoid numerical approximation errors.

Let $p(x) \in \mathbb{R}[x]$ be a nonzero polynomial and let $\mu$ be a complex number.
Define $\mult(\mu, p)$ to be the largest nonnegative integer $m$ such that $(x-\mu)^m$ divides $p(x)$.
By definition, $\mult(\mu, p) = 0$ means that $\mu$ is not a root of $p(x)$ or equivalently, $p(\mu) \neq 0$.
Additionally, if $\mu$ is a root of $p(x)$ then $\mult(\mu, p)$ is also the multiplicity of $\mu$ as a root of $p(x)$.
If $\mult(\mu, p)=1$ then we call $\mu$ a \textbf{simple root} of $p(x)$.
If $\mult(\mu, p)>1$ then we call $\mu$ a \textbf{multiple root} of $p(x)$.
In particular, if $\mult(\mu, p)=2$ then we call $\mu$ a \textbf{double root} of $p(x)$.
Let $p^\prime(x)$ denote the derivative of $p(x)$.

\begin{lem} \label{lem:dervmult}
    Let $p(x) \in \mathbb{R}[x]$ be a non-constant polynomial and let $\mu$ be a root of $p(x)$.
    Then $\mult(\mu, p) = \mult(\mu, p^\prime)+1$.
\end{lem}

\begin{proof}
    Note that there exists a positive integer $m$ such that $\mult(\mu, p) = m$.
    It follows that $p(x) = (x-\mu)^m \cdot q(x)$ for some polynomial $q(x) \in \mathbb{C}[x]$ where $q(\mu) \neq 0$.
    Thus, we obtain
    \[
    p^\prime(x) = (x-\mu)^{m-1} \left( m \cdot q(x) + (x-\mu) \cdot q^\prime(x) \right).
    \]
    Observe that $m \cdot q(\mu) + (\mu-\mu) \cdot q^\prime(\mu) = m \cdot q(\mu) \neq 0$.
    This implies that $\mult(\mu, p^\prime) = m-1$.
\end{proof}

We use the next lemma to prove Theorem~\ref{thm:realrootedalgo}.

\begin{lem} \label{lem:realrootedcond}
    Let $p(x) \in \mathbb{R}[x]$ be a non-constant real-rooted polynomial.
    Let $\displaystyle P(x) = \int_0^x p(y) \dd{y}$ and let $C \in \mathbb{R}$.
    Then $P(x)+C$ is real-rooted if and only if for all roots $\mu$ of $p(x)$, we have 
    $P(\mu)+C \geqslant 0$ if $p^\prime(\mu) < 0$, $P(\mu)+C \leqslant 0$ if $p^\prime(\mu) > 0$, and $P(\mu)+C = 0$ if $p^\prime(\mu) = 0$.
\end{lem}

\begin{proof}
    Let $n \geqslant 1$ be the degree of $p(x)$.
    Let $\mu_1 \leqslant \mu_2 \leqslant \dots \leqslant \mu_n$ be the roots of $p(x)$.
    Define $Q(x) \coloneqq P(x)+C$.
    Suppose that $Q(x)$ is real-rooted with roots $\lambda_1 \leqslant \lambda_2 \leqslant \dots \leqslant \lambda_{n+1}$.
    Observe that the roots of $p(x)$ interlace the roots of $Q(x)$, that is,
    \[
    \lambda_1 \leqslant \mu_1 \leqslant \lambda_2 \leqslant \dots \leqslant \mu_n \leqslant \lambda_{n+1}.
    \]
    Suppose that $p^\prime(\mu_i)=0$ for some $i \in \{1,\dots,n\}$.
    By Lemma~\ref{lem:dervmult}, we have that $\mu_i$ is a multiple root of $p(x)$.
    Then $\mu_i = \mu_{i+1}$ or $\mu_{i-1} = \mu_i$, and hence $\mu_i = \lambda_{i+1}$ or $\mu_i = \lambda_i$, respectively.
    In either case, it follows that $Q(\mu_i)=0$.
    Next, suppose that $p^\prime(\mu_i) > 0$ for some $i \in \{1,\dots,n\}$.
    We will show that $Q(\mu_i) \leqslant 0$.
    For the sake of contradiction, assume that $Q(\mu_i) > 0$.
    Note that $Q(\lambda_i)=Q(\lambda_{i+1})=0$ so we have $\lambda_i < \mu_i < \lambda_{i+1}$.
    By the mean value theorem, there exists $a \in (\mu_i,\lambda_{i+1})$ such that $Q^\prime(a) = p(a) < 0$.
    On the other hand, there exists $\varepsilon >0$ small enough such that $\mu_i+\varepsilon < a$ and $p(\mu_i+\varepsilon) >0$.
    The intermediate value theorem implies that there exists $b \in (\mu_i+\varepsilon,a) \subseteq (\mu_i,\lambda_{i+1}) \subseteq (\mu_i,\mu_{i+1})$ such that $p(b)=0$, which is a contradiction.
    Using a similar argument, we find that if $p^\prime(\mu_i) < 0$ for some $i \in \{1,\dots,n\}$, then $Q(\mu_i) \geqslant 0$.

    Conversely, suppose that, for all $i \in \{1,\dots,n\}$, we have $Q(\mu_i) \geqslant 0$ if $p^\prime(\mu_i) < 0$, $Q(\mu_i) \leqslant 0$ if $p^\prime(\mu_i) > 0$, and $Q(\mu_i) = 0$ if $p^\prime(\mu_i) = 0$.
    Define the intervals: $I_1 \coloneqq (-\infty, \mu_1], I_i \coloneqq [\mu_{i-1}, \mu_i]$ for all integers $2 \leqslant i \leqslant n$, and $I_{n+1} \coloneqq [\mu_n, \infty)$.
    To prove that $Q(x)$ is real-rooted, we will show that there exists a function $\lambda: \{I_1, I_2, \dots, I_{n+1}\} \to \mathbb{R}$ such that
    \[
    Q(x) = q \prod_{i=1}^{n+1} \left( x-\lambda(I_i) \right)
    \]
    for some nonzero real number $q$.
    First, suppose that $\mu$ is a multiple root of $p(x)$ such that $\mult(\mu, p) = m$, where $2 \leqslant m \leqslant n$.
    It follows that $p^\prime(\mu) = 0$ and there exists $j \in \{1,2,\dots,n-m+1\}$ such that
    \[
    \mu = \mu_j = \mu_{j+1} = \dots = \mu_{j+m-1}.
    \]
    By our assumption, we have $Q(\mu) = 0$, so $\mu$ is a root of $Q(x)$.
    By Lemma~\ref{lem:dervmult}, we must have $\mult(\mu, Q) = m+1$.
    Hence, we let $\lambda(I_i) = \mu$ for all $i \in \{j, j+1, \dots, j+m\}$.
    Next, suppose that for some $1 \leqslant i \leqslant n$, we have that $\mu_i$ is a simple root of $p(x)$ and $Q(\mu_i)=0$.
    Then, by Lemma~\ref{lem:dervmult}, we have that $\mu_i$ is a double root of $Q(x)$ so we let $\lambda(I_i) = \lambda(I_{i+1}) = \mu_i$.

    Suppose that for some $1 \leqslant i \leqslant n-1$, we have $\mu_i < \mu_{i+1}$.
    We want to prove that there exists a unique $\lambda^\star \in I_{i+1}$ such that $Q(\lambda^\star)=0$.
    Hence, we can let $\lambda(I_{i+1})=\lambda^\star$.
    We first show that there exists at least one root of $Q(x)$ in $I_{i+1}$.
    By the argument above, this is true if at least one of $\mu_i$ and $\mu_{i+1}$ is a multiple root of $p(x)$.
    Hence, suppose that $\mu_i$ and $\mu_{i+1}$ are both simple roots of $p(x)$.
    We also assume that $Q(\mu_i) \neq 0$ and $Q(\mu_{i+1}) \neq 0$.
    Since $\mu_i$ and $\mu_{i+1}$ are simple roots of $p(x)$, then $p^\prime(\mu_i) \neq 0$ and $p^\prime(\mu_{i+1}) \neq 0$.
    We claim that $p^\prime(\mu_i)$ and $p^\prime(\mu_{i+1})$ must have the opposite signs.
    Otherwise, we first suppose that $p^\prime(\mu_i)$ and $p^\prime(\mu_{i+1})$ are both positive.
    Then there exist $\varepsilon_1,\varepsilon_2 > 0$ small enough such that $p(\mu_i+\varepsilon_1)>0$, $p(\mu_{i+1}-\varepsilon_2)<0$, and $\mu_i+\varepsilon_1 < \mu_{i+1}-\varepsilon_2$.
    By the intermediate value theorem, there exists $a \in (\mu_i+\varepsilon_1,\mu_{i+1}-\varepsilon_2) \subseteq (\mu_i,\mu_{i+1})$ such that $p(a)=0$, which is a contradiction.
    Similarly, assuming that both $p^\prime(\mu_i)$ and $p^\prime(\mu_{i+1})$ are negative will lead to a contradiction.
    Thus, $p^\prime(\mu_i)$ and $p^\prime(\mu_{i+1})$ must have the opposite signs.
    By our assumption, this implies that $Q(\mu_i)$ and $Q(\mu_{i+1})$ also have the opposite signs.
    Hence, by the intermediate value theorem, the interval $(\mu_i, \mu_{i+1})$ contains at least one root of $Q(x)$.
    Next, we will prove that there exists at most one root of $Q(x)$ in $I_{i+1}$.
    Suppose instead that there exist distinct $a, b \in I_{i+1}$ such that $Q(a)=Q(b)=0$.
    By Rolle's theorem, there will be another root of $p(x)$ in $(\mu_i, \mu_{i+1})$, which is a contradiction.
    Therefore, there exists precisely one root of $Q(x)$ in $I_{i+1}$.
    
    Lastly, consider $I_1$ and $I_{n+1}$.
    Suppose that $\mu_1$ is a simple root of $p(x)$ and $Q(\mu_1) \neq 0$.
    Then $p^\prime(\mu_1) \neq 0$ so suppose that $p^\prime(\mu_1)>0$.
    This means that $Q(x)$ has a local minimum at $x=\mu_1$.
    By our assumption, we also have $Q(\mu_1)<0$.
    Hence, there exists a unique real number $\lambda^\star$ less than $\mu_1$ such that $Q(\lambda^\star)=0$.
    Thus, we let $\lambda(I_1) = \lambda^\star$.
    The argument is similar for $p^\prime(\mu_1)<0$.
    Likewise, if $\mu_n$ is a simple root of $p(x)$ and $Q(\mu_n) \neq 0$, there also exists a unique real number $\lambda^\star$ greater than $\mu_n$ such that $Q(\lambda^\star)=0$.
    Hence, we let $\lambda(I_{n+1}) = \lambda^\star$.
    Combining all of these cases above, we can construct a well-defined function $\lambda$.
    Therefore, we conclude that $Q(x)$ is real-rooted.
\end{proof}

Let $P(x) \in \mathbb{R}[x]$ be a non-constant polynomial.
Note that if the degree of $P(x)$ is odd, then it has an equal number of local maximum and local minimum points.
Meanwhile, if the degree of $P(x)$ is even and it has a positive leading coefficient, then the number of local minimum points of $P(x)$ is exactly one more than the number of its local maximum points.

\begin{thm} \label{thm:realrootedalgo}
    Let $p(x) \in \mathbb{R}[x]$ be a real-rooted polynomial of degree $n \geqslant 2$ with a positive leading coefficient, and let $\mu_1 \leqslant \mu_2 \leqslant \dots \leqslant \mu_n$ be its roots.
    Let $\displaystyle P(x)=\int_0^x p(y) \dd{y}$.
    Let $h_1 \leqslant \dots \leqslant h_n$ be real numbers such that the multisets $\{h_1, h_2, \dots, h_n\}$ and $\{-P(\mu_1), -P(\mu_2), \dots, -P(\mu_n)\}$ are equal.
    Let $\displaystyle k = \left\lfloor \frac{n}{2} \right\rfloor$ and let $h$ be a real number such that $h_k \leqslant h \leqslant h_{k+1}$.
    Suppose that $P(x)+h$ is real-rooted and let $C$ be a real number.
    Then $P(x)+C$ is real-rooted if and only if $h_k \leqslant C \leqslant h_{k+1}$.
\end{thm}

\begin{proof}
    Since the leading coefficient of $p(x)$ is positive, then the leading coefficient of $P(x)$ is also positive.
    Let $u$ be the number of local maximum points of $P(x)$ and let $d$ be the number of local minimum points of $P(x)$.
    Observe that $u = d \leqslant k$ if $n=2k$, and $u = d-1 \leqslant k$ if $n=2k+1$.
    Equalities hold precisely when all roots of $p(x)$ are distinct.
    Let $I =\{1,2,\dots,n\}$.
    Let $J \subseteq I$ such that $i \in J$ if and only if $p^\prime(\mu_i) < 0$.
    Clearly, we have $\lvert J \rvert \leqslant u \leqslant k$.
    Let $K \subseteq I$ such that $i \in K$ if and only if $p^\prime(\mu_i) > 0$.
    Clearly, we have $\lvert K \rvert \leqslant d \leqslant n-k$.

    First, suppose that $P(x)+C$ is real-rooted.
    We will prove that $h_k \leqslant C \leqslant h_{k+1}$ without having to use the assumption that $P(x)+h$ is real-rooted.
    Suppose that $n = 2k$.
    Observe that $\lvert J \rvert \leqslant k$ and $\lvert K \rvert \leqslant k$.
    We then partition $I$ into two subsets $I_1$ and $I_2$ such that $J \subseteq I_1$, $K \subseteq I_2$, and $\lvert I_1 \rvert = \lvert I_2 \rvert = k$.
    Let $i \in \left(I_1 \backslash J\right) \cup \left(I_2 \backslash K\right)$.
    Observe that $p^\prime(\mu_i) = 0$.
    Since $P(x)+C$ is real-rooted, by Lemma~\ref{lem:realrootedcond}, we have that $P(\mu_i)+C = 0$.
    Clearly, both $C \geqslant -P(\mu_i)$ and $C \leqslant -P(\mu_i)$ are true.
    Again, by Lemma~\ref{lem:realrootedcond}, we also have that $C \geqslant -P(\mu_i)$ for all $i \in J$, and $C \leqslant -P(\mu_i)$ for all $i \in K$.
    It follows that $C \geqslant -P(\mu_i)$ for all $i \in I_1$, and $C \leqslant -P(\mu_i)$ for all $i \in I_2$.
    Therefore, we conclude that $h_k \leqslant C \leqslant h_{k+1}$.
    A similar argument can be applied to the case where $n$ is odd.
    Suppose that $n = 2k+1$.
    Observe that $\lvert J \rvert \leqslant k$ and $\lvert K \rvert \leqslant k+1$.
    We then partition $I$ into two subsets $I_1$ and $I_2$ such that $J \subseteq I_1$, $K \subseteq I_2$, $\lvert I_1 \rvert = k$, and $\lvert I_2 \rvert = k+1$.
    For all $i \in \left(I_1 \backslash J\right) \cup \left(I_2 \backslash K\right)$, we have $p^\prime(\mu_i) = 0$.
    Applying Lemma~\ref{lem:realrootedcond}, we deduce that $C \geqslant -P(\mu_i)$ for all $i \in I_1$, and $C \leqslant -P(\mu_i)$ for all $i \in I_2$.
    Therefore, we conclude that $h_k \leqslant C \leqslant h_{k+1}$.

    Conversely, suppose that $h_k \leqslant C \leqslant h_{k+1}$.
    Note that if $h_k = h_{k+1}$ then $C=h=h_k=h_{k+1}$, which clearly implies that $P(x)+C$ is real-rooted.
    It remains to consider the case when $h_k < h_{k+1}$.
    Let $\mu$ be a root of $p(x)$.
    There exists a bijection $\varphi$ from $I$ to itself such that for all $i \in I$, we have $h_i = -P(\mu_{\varphi(i)})$.
    In particular, there exist $s, t \in I$ such that $\mu = \mu_t$ and $\varphi(s)=t$, which means that $h_s = -P(\mu_t) = -P(\mu)$.
    Define $I_{\text{left}} \coloneqq \{1,2,\dots,k\}$ and $I_{\text{right}} \coloneqq I \setminus I_{\text{left}}$.
    We will first show that $p^\prime(\mu) \neq 0$.
    For the sake of contradiction, assume instead that $p^\prime(\mu) = 0$.
    By Lemma~\ref{lem:dervmult}, we have that $\mult(\mu,p) \geqslant 2$.
    Since $P(x)+h$ is real-rooted, by Lemma~\ref{lem:realrootedcond}, we have that $P(\mu)+h = 0$.
    This implies that $h = h_k$ or $h = h_{k+1}$.
    Suppose first that $h = h_k$.
    Let $j \in I_{\text{right}}$.
    Then
    \begin{equation} \label{eq:varphi}
        -P(\mu_{\varphi(j)}) = h_j > h_k = h \implies P(\mu_{\varphi(j)})+h <0.
    \end{equation}
    By Lemma~\ref{lem:realrootedcond}, we obtain $p^\prime(\mu_{\varphi(j)}) > 0$.
    Hence, we have $\varphi(I_{\text{right}}) \subseteq K$, which further implies that $d = n-k$ and $u=k$.
    However, since $\mult(\mu,p) \geqslant 2$, we deduce that $h_k = h_{k-1} = -P(\mu)$.
    This yields $u < k$, which is a contradiction.
    In a similar manner, we can derive both $d=n-k$ and $d<n-k$ if $h = h_{k+1}$, which is a contradiction.
    Either case of $h=h_k$ or $h=h_{k+1}$ yields a contradiction, so we conclude that $p^\prime(\mu) \neq 0$.
    
    Next, suppose that $p^\prime(\mu) > 0$.
    We want to prove that $P(\mu)+C \leqslant 0$.
    By Lemma~\ref{lem:realrootedcond}, we obtain $P(\mu)+h \leqslant 0$.
    Note that $h_s = -P(\mu) \geqslant h \geqslant h_k$.
    We will prove that $h_s \geqslant h_{k+1}$.
    Assume otherwise that $h_s < h_{k+1}$.
    Then $h_k \geqslant h_s$, which means that $h = h_s = h_k = -P(\mu)$.
    Hence, we must have $t \notin \varphi(I_{\text{right}})$.
    By \eqref{eq:varphi} and Lemma~\ref{lem:realrootedcond}, we obtain $p^\prime(\mu_{\varphi(j)}) > 0$ for all $j \in I_{\text{right}}$.
    Since $p^\prime(\mu_t) = p^\prime(\mu) > 0$, we find that $\lvert K \rvert \geqslant n-k+1$, which is a contradiction.
    Therefore, $h_s \geqslant h_{k+1}$ and the desired inequality follows directly from $-P(\mu) = h_s \geqslant h_{k+1} \geqslant C$.
    Similarly, if $p^\prime(\mu) < 0$ then $h_s \leqslant h_k$, which immediately implies that $P(\mu)+C \geqslant 0$.
    Altogether, applying Lemma~\ref{lem:realrootedcond}, we conclude that $P(x)+C$ is real-rooted.
\end{proof}

Let $p(x), q(x) \in \mathbb{R}[x]$ be nonzero polynomials such that $\deg p(x) = d \geqslant 0$ and $\deg q(x) = e \geqslant 0$.
Let $a \in \mathbb{R}$ and let $\mu_1, \mu_2, \dots, \mu_d$ be complex numbers such that $\displaystyle p(x) = a \prod_{i=1}^d (x-\mu_i)$.
Similarly, let $b \in \mathbb{R}$ and let $\lambda_1, \lambda_2, \dots, \lambda_e$ be complex numbers such that $\displaystyle q(x) = b \prod_{j=1}^e (x-\lambda_j)$.
We use the following definition of the \textbf{resultant} of $p(x)$ and $q(x)$:
\[
\Res_x(p(x), q(x)) \coloneqq a^e b^d \prod_{\substack{1 \leqslant i \leqslant d, \mathstrut \\ 1 \leqslant j \leqslant e}} (\mu_i-\lambda_j).
\]
Moreover, we define the \textbf{discriminant} of $p(x)$ as follows:
\[
\Disc_x(p(x)) \coloneqq \frac{(-1)^{\frac{d(d-1)}{2}}}{a} \cdot \Res_x(p(x), p^\prime(x)).
\]
For more details on resultant and discriminant, we refer the reader to \cite{GKZ94}.

An important contribution of McKee and Smyth~\cite{MS04} was to use the discriminant to speed up the computation of the multiset $\{-P(\mu_1), -P(\mu_2), \dots, -P(\mu_n)\}$.
In Proposition~\ref{prop:discroots} below, we provide an explicit formula for the discriminant of any antiderivative of a given nonzero polynomial from $\mathbb{R}[x]$.

\begin{prop} \label{prop:discroots}
    Let $p(x) \in \mathbb{R}[x]$ be a nonzero polynomial of degree $n \geqslant 0$.
    Let $a \in \mathbb{R}$ and let $\mu_1,\mu_2,\dots,\mu_n$ be complex numbers such that $\displaystyle p(x) = a \prod_{i=1}^n (x-\mu_i)$.
    Let $\displaystyle P(x) = \int_0^x p(y) \dd{y}$ and let $C \in \mathbb{R}$.
    Then, we have
    \[
    \Disc_x(P(x)+C) = (-1)^{\frac{n(n+1)}{2}} \cdot a^n (n+1) \cdot \prod_{i=1}^n (P(\mu_i)+C).
    \]
\end{prop}

\begin{proof}
    By definition, we have
    \[
    \Disc_x(P(x)+C) = (-1)^{\frac{n(n+1)}{2}} \cdot \frac{n+1}{a} \cdot \Res_x (P(x)+C, p(x)).
    \]
    Suppose that $\displaystyle P(x)+C = \frac{a}{n+1} \prod_{j=1}^{n+1} (x-\lambda_j)$ where $\lambda_j \in \mathbb{C}$ for all $j \in \{1,2,\dots,n+1\}$.
    Observe that
    \begin{align*}
        \Res_x (P(x)+C, p(x)) &= \Res_x (p(x), P(x)+C) = a^{n+1} \cdot {\left( \frac{a}{n+1} \right)}^n \cdot \prod_{i=1}^n \prod_{j=1}^{n+1} (\mu_i - \lambda_j) \\
        &= a^{n+1} \cdot \prod_{i=1}^n \left( \frac{a}{n+1} \prod_{j=1}^{n+1} (\mu_i - \lambda_j) \right) = a^{n+1} \cdot \prod_{i=1}^n (P(\mu_i)+C).
    \end{align*}
    Therefore, it follows that
    \[
    \Disc_x(P(x)+C) = (-1)^{\frac{n(n+1)}{2}} \cdot a^n (n+1) \cdot \prod_{i=1}^n (P(\mu_i)+C). \qedhere
    \]
\end{proof}

If we consider $\Disc_x(P(x)+C)$ as a univariate polynomial in $C$, then the multiset of its roots is precisely the multiset $\{-P(\mu_1), -P(\mu_2), \dots, -P(\mu_n)\}$.

Let $p(x) \in \mathbb{R}[x]$ be a nonzero polynomial of degree $n \geqslant 0$.
We define
\begin{equation} \label{eq:minpx}
    \Min(p, x) \coloneqq \frac{p(x)}{\gcd(p(x), p^\prime(x))}.
\end{equation}
Note that $\displaystyle \Min(p, x) \in \mathbb{R}[x]$ is a monic polynomial of degree $r$ for some $0 \leqslant r \leqslant n$.
The set of complex roots of $p(x)$ is identical to the set of complex roots of $\Min(p, x)$.
Moreover, each root of $\Min(p, x)$ is simple.
Then $p(x)$ is real-rooted if and only if $\Min(p, x)$ is real-rooted.
In \hyperref[alg:realrootedcheck]{$\mathsf{IsRealRooted}$}, we apply \emph{Sturm's theorem}~\cite[Theorem 2.50]{BPR06} on $\Min(p, x)$ to determine whether $p(x)$ is real-rooted without finding all of its roots.
First, we form the \emph{Sturm sequence} $q_0(x), q_1(x), \dots, q_s(x)$ of $\Min(p, x)$ where $s \leqslant r$.
We set $q_0(x) = \Min(p, x)$ and $q_1(x) = {q_0}^{\prime}(x)$.
For each $2 \leqslant i \leqslant s$, we let $-q_i(x) = \rem(q_{i-2}(x), q_{i-1}(x))$, which is the remainder of the \emph{Euclidean division} of $q_{i-2}(x)$ by $q_{i-1}(x)$.
The sequence terminates at $q_s(x)$, the first time we encounter a zero remainder $\rem(q_{s-1}(x), q_s(x)) = 0$.
If $r=0$ then the Sturm sequence of $\Min(p, x)$ consists of only one polynomial $q_0(x) = \Min(p, x)$.
Let $\sgn$ be the sign function, and let $\mathbf{v} = (v_0, v_1, \dots, v_s) \in \{\pm 1\}^{s+1}$ such that for all $0 \leqslant i \leqslant s$, we have $\displaystyle v_i = \lim_{x \to -\infty} \sgn(q_i(x))$.
If the degree of $q_i(x)$ is even, then $v_i$ is equal to the sign of the leading coefficient of $q_i(x)$.
If the degree of $q_i(x)$ is odd, then $v_i$ is equal to the negative of the sign of the leading coefficient of $q_i(x)$.
Similarly, let $\mathbf{w}= (w_0, w_1, \dots, w_s) \in \{\pm 1\}^{s+1}$ such that for all $0 \leqslant i \leqslant s$, we have $\displaystyle w_i = \lim_{x \to \infty} \sgn(q_i(x))$.
This means that $w_i$ is equal to the sign of the leading coefficient of $q_i(x)$.
Define $\Var(-\infty)$ and $\Var(+\infty)$ to be the numbers of sign changes in $\mathbf{v}$ and $\mathbf{w}$, respectively.
By Sturm's theorem, the number of real roots of $\Min(p, x)$ is equal to $\Var(-\infty)-\Var(+\infty)$.

\begin{algorithm}[htbp]
\capstart
\caption{$\mathsf{IsRealRooted}(p(x))$}
\label{alg:realrootedcheck}
\SetKwFunction{True}{True}
\SetKwFunction{False}{False}
\SetKwInOut{Input}{Input}
\SetKwInOut{Output}{Output}
\Input{A nonzero polynomial $p(x) \in \mathbb{R}[x]$.}
\Output{\True if $p(x)$ is real-rooted, and \False otherwise.}
    $\displaystyle \Min(p, x) \gets \frac{p(x)}{\gcd(p(x), p^\prime(x))}$\;
    Generate the Sturm sequence of $\Min(p, x)$\;
    Determine $\mathbf{v}$ and $\mathbf{w}$\;
    Compute $\Var(-\infty)$ and $\Var(+\infty)$\;
    \eIf{$\Var(-\infty)-\Var(+\infty) = \deg \Min(p, x)$}
    {\Return{\True}\;}{\Return{\False}\;}
\end{algorithm}

\hyperref[alg:numrealrooted]{$\mathsf{AllRealRooted}$} enumerates all real-rooted integer polynomials with fixed first few (at least three) leading coefficients.
Observe that if we are given just the first two leading coefficients, the number of the corresponding real-rooted integer polynomials could be infinite.
We will show that if we know the first few (at least three) leading coefficients, then the number of the corresponding real-rooted integer polynomials is finite.
Let $n \geqslant 2$ be an integer and let $t$ be an integer such that $3 \leqslant t \leqslant n+1$.
Let $\bm{v} \in \mathbb{Z}^t$ be a vector and for each $i \in \{1, 2, \dots, t\}$, let $\bm{v}(i)$ denote the $i$th entry of $\bm{v}$.
Suppose that $\bm{v}(1) \neq 0$.
Define the set $T(n,t,\bm{v})$ as the set of all real-rooted integer polynomials $\displaystyle p(x) = \sum_{i=0}^n a_i x^{n-i}$ of degree $n$ such that $a_{i-1} = \bm{v}(i)$ for all $i \in \{1, 2, \dots, t\}$.
The next proposition states that the cardinality of $T(n,t,\bm{v})$ is finite.

\begin{prop} \label{prop:finite}
    Let $n \geqslant 2$ be an integer and let $t$ be an integer such that $3 \leqslant t \leqslant n+1$.
    Let $\bm{v} \in \mathbb{Z}^t$ be a vector such that $\bm{v}(1) \neq 0$.
    Then $T(n,t,\bm{v})$ is a finite set.
\end{prop}

\begin{proof}
    Let $p(x) \in T(n,t,\bm{v})$.
    Let $a_0, a_1, \dots, a_n$ be integers and let $\lambda_1, \lambda_2, \dots, \lambda_n$ be real numbers such that
    \[
    p(x) = \sum_{i=0}^n a_i x^{n-i} = a_0 \prod_{i=1}^n (x-\lambda_i).
    \]
    By Newton's identities, for all $i \in \{1, 2, \dots, n\}$, we have $\displaystyle \left\lvert \lambda_i \right\rvert \leqslant \sqrt{\frac{a_1^2-2 a_0 a_2}{a_0^2}}$.
    Clearly, we have $T(n,t,\bm{v}) = \varnothing$ if $a_1^2 < 2 a_0 a_2$.
    Otherwise, the coefficient $a_i$ is bounded for all $t \leqslant i \leqslant n$.
    Therefore, we conclude that $T(n,t,\bm{v})$ is a finite set.
\end{proof}

We will now describe \hyperref[alg:numrealrooted]{$\mathsf{AllRealRooted}$}.
The inputs are an integer polynomial $\displaystyle \hat{p}(x) = \sum_{i=0}^n \hat{a}_i x^{n-i}$ where $n \geqslant 2$ and $\hat{a}_0>0$, together with an integer $t$ where $3 \leqslant t \leqslant n+1$.
The output is the set $T(n, t, \bm{v})$ where $\bm{v} \in \mathbb{Z}^t$ and $\bm{v}(i) = \hat{a}_{i-1}$ for all $i \in \{1,2,\dots, t\}$.
Let $p(x) \in T(n, t, \bm{v})$ and let $a_0, a_1, \dots, a_n$ be integers such that $\displaystyle p(x) = \sum_{i=0}^n a_i x^{n-i}$.
In particular, we have that $a_{i-1} = \bm{v}(i)$ for all $i \in \{1, 2, \dots, t\}$.
For each $i \in \{t-1, t, \dots, n\}$, we define the polynomial
\begin{equation} \label{eq:dervpoly}
    p_i(x) \coloneqq \dv[n-i]{x} p(x) = \sum_{j=0}^i \frac{(n-j)!}{(i-j)!} \cdot a_j x^{i-j}.
\end{equation}
Note that $p_n(x) = p(x)$.
For all $i \in \{t, t+1, \dots, n\}$, we also have the following relation
\begin{equation} \label{eq:consecderv}
    p_i(x) = \int_0^x p_{i-1}(y) \dd{y} + (n-i)! \cdot a_i.
\end{equation}
By \eqref{eq:dervpoly}, we have that the $(n-t+1)$th derivative of $p(x)$ is equal to $p_{t-1}(x)$, which is also equal to the $(n-t+1)$th derivative of $\hat{p}(x)$.
Since $p(x)$ is real-rooted, $p_{t-1}(x)$ must also be real-rooted.
Throughout \hyperref[alg:numrealrooted]{$\mathsf{AllRealRooted}$}, we run \hyperref[alg:realrootedcheck]{$\mathsf{IsRealRooted}$} to check whether any given nonzero integer polynomial is real-rooted.
If we check that $p_{t-1}(x)$ is not real-rooted, then we return $T(n, t, \bm{v}) = \varnothing$.
Otherwise, if $p_{t-1}(x)$ is real-rooted, starting from $i=t$ to $i=n$, we will determine all possible $p_i(x)$ for each possible $p_{i-1}(x)$ by using Theorem~\ref{thm:realrootedalgo}.
Suppose that for some $i \in \{t, t+1, \dots, n\}$, we have determined $p_{i-1}(x)$, which is real-rooted.
Let $\displaystyle P(x) = \int_0^x p_{i-1}(y) \dd{y}$.
Thus, as a univariate polynomial in $C$, $\displaystyle \Disc_x(P(x)+C)$ is real-rooted by Proposition~\ref{prop:discroots}.
Denote the roots by $h_1 \leqslant h_2 \leqslant \dots \leqslant h_{i-1}$.
Applying Theorem~\ref{thm:realrootedalgo}, we narrow down the range of $C$ to the closed interval $[h_k, h_{k+1}]$, where $\displaystyle k = \left\lfloor \frac{i-1}{2} \right\rfloor$.
By \eqref{eq:consecderv}, all possible values of $C$ are those integers in the closed interval $[h_k, h_{k+1}]$ that are divisible by $(n-i)!$.

\begin{remark}
    \label{rem:discNumerical}
    Since $C$ is an integer, one can compute the integer endpoints $\lceil h_k \rceil$ and $\lfloor h_{k+1} \rfloor$ without the need for floating-point arithmetic.
    This observation is a key insight due to McKee and Smyth~\cite{MS04}.
\end{remark}

In Line~\ref{numsolve} of \hyperref[alg:numrealrooted]{$\mathsf{AllRealRooted}$}, we use numerical approximations $\tilde{h}_k$ and $\tilde{h}_{k+1}$ of $h_k$ and $h_{k+1}$, respectively.
We assume that the absolute error of our numerical approximations is less than $1/2$.
We will apply Lemma~\ref{lem:omegafloorceil}, which we incorporate in \hyperref[alg:endpoints]{$\mathsf{EndPoints}$}.
We implement \hyperref[alg:endpoints]{$\mathsf{EndPoints}$} in the context of determining $\lceil h_k \rceil$ and $\lfloor h_{k+1} \rfloor$ for \hyperref[alg:numrealrooted]{$\mathsf{AllRealRooted}$} with the help of Lemma~\ref{lem:omegafloorceil}.
As previously noted in Remark~\ref{remark:round}, we use the more general $\mathsf{Round}(x)$ instead of strictly just $\lceil x \rfloor$.
The inputs of \hyperref[alg:endpoints]{$\mathsf{EndPoints}$} are a polynomial $P(x) \in \mathbb{R}[x]$ of degree at least three, together with two real numbers $\tilde{\omega}_1$ and $\tilde{\omega}_2$.
First, we initialise the output set as the empty set.
If $P(x)+\mathsf{Round}(\tilde{\omega}_1)$ is real-rooted, then we add $\mathsf{Round}(\tilde{\omega}_1)$ to the output set.
Otherwise, if $P(x)+\mathsf{Round}(\tilde{\omega}_1)+1$ is real-rooted, then we add $\mathsf{Round}(\tilde{\omega}_1)+1$ to the output set.
Next, if $P(x)+\mathsf{Round}(\tilde{\omega}_2)$ is real-rooted, then we add $\mathsf{Round}(\tilde{\omega}_2)$ to the output set.
Otherwise, if $P(x)+\mathsf{Round}(\tilde{\omega}_2)-1$ is real-rooted, then we add $\mathsf{Round}(\tilde{\omega}_2)-1$ to the output set.
Note that the output set will consist of at most two integers.

\begin{algorithm}[htbp]
\capstart
\caption{$\mathsf{EndPoints}(P(x), \tilde{\omega}_1, \tilde{\omega}_2)$}
\label{alg:endpoints}
\SetKwData{out}{out}
\SetKwFunction{True}{True}
\SetKwFunction{False}{False}
\SetKwInOut{Input}{Input}
\SetKwInOut{Output}{Output}
\Input{A polynomial $P(x) \in \mathbb{R}[x]$ of degree at least three, and two real numbers $\tilde{\omega}_1$, $\tilde{\omega}_2$.}
\Output{A set consisting of at most two integers.}
    $\out \gets \varnothing$\;
    
    \uIf{\hyperref[alg:realrootedcheck]{$\mathsf{IsRealRooted}$}$\left( P(x)+\mathsf{Round}(\tilde{\omega}_1) \right) = \True$}{Append $\mathsf{Round}(\tilde{\omega}_1)$ to \out\;}
    \ElseIf{\hyperref[alg:realrootedcheck]{$\mathsf{IsRealRooted}$}$\left( P(x)+\mathsf{Round}(\tilde{\omega}_1) + 1 \right) = \True$}{Append $\mathsf{Round}(\tilde{\omega}_1)+1$ to \out\;}
    
    \uIf{\hyperref[alg:realrootedcheck]{$\mathsf{IsRealRooted}$}$\left( P(x)+\mathsf{Round}(\tilde{\omega}_2) \right) = \True$}{Append $\mathsf{Round}(\tilde{\omega}_2)$ to \out\;}
    \ElseIf{\hyperref[alg:realrootedcheck]{$\mathsf{IsRealRooted}$}$\left( P(x)+\mathsf{Round}(\tilde{\omega}_2) - 1 \right) = \True$}{Append $\mathsf{Round}(\tilde{\omega}_2)-1$ to \out\;}

    \Return{\out}\;
\end{algorithm}

In Line~\ref{callendp} of \hyperref[alg:numrealrooted]{$\mathsf{AllRealRooted}$}, we run $\hyperref[alg:endpoints]{\mathsf{EndPoints}}(P(x), \tilde{h}_k, \tilde{h}_{k+1})$.
If the output is the empty set, then by Theorem~\ref{thm:realrootedalgo}, we can continue directly to the next possible $p_{i-1}(x)$.
Otherwise, we set $\lceil h_k \rceil$ and $\lfloor h_{k+1} \rfloor$ as the minimum and the maximum of the output set, respectively.
Lastly, in Line~\ref{rangeC}, we generate all possible real-rooted $P(x)+C$ where the range of $C$ is all of the integers from $\lceil h_k \rceil$ to $\lfloor h_{k+1} \rfloor$ that are divisible by $(n-i)!$.

\begin{algorithm}[htbp]
\capstart
\caption{$\mathsf{AllRealRooted}(\hat{p}(x), t)$}
\label{alg:numrealrooted}
\SetKwData{out}{out}
\SetKwData{points}{points}
\SetKwData{temp}{temp}
\SetKwInOut{Input}{Input}
\SetKwInOut{Output}{Output}
\SetKw{KwBy}{by}
\Input{An integer polynomial $\displaystyle \hat{p}(x) = \sum_{i=0}^n \hat{a}_i x^{n-i}$ where $n \geqslant 2$ and $\hat{a}_0>0$, and an integer $t \in \{3,4,\dots,n+1\}$.}
\Output{The set $T(n, t, \bm{v})$ where $\bm{v} \in \mathbb{Z}^t$ and $\bm{v}(i) = \hat{a}_{i-1}$ for all $i \in \{1,2,\dots, t\}$.}
    Initialise each of \temp and \out as an empty array\;
    $p_{t-1}(x) \gets$ $(n-t+1)$th derivative of $\hat{p}(x)$\;\label{arrp2x}
    \eIf{\hyperref[alg:realrootedcheck]{$\mathsf{IsRealRooted}$}$(p_{t-1}(x)) = \True$\label{arrstart}}
    {
    Append $p_{t-1}(x)$ to \out\;
    \For{$i = t$ \KwTo $n$}
        {
        \For{$j = 1$ \KwTo $\#\out$}
            {
            $\displaystyle k \gets \left\lfloor \frac{i-1}{2} \right\rfloor$, $p(x) \gets$ the $j$th entry of \out, and $\displaystyle P(x) \gets \int_0^x p(y) \dd{y}$\;
            
            Find numerical approximations $\tilde{h}_k$ and $\tilde{h}_{k+1}$ of $h_k$ and $h_{k+1}$, respectively, where $h_1 \leqslant h_2 \leqslant \dots \leqslant h_{i-1}$ are all roots of $\Disc_x(P(x)+C)$ as a univariate polynomial in $C$\; \label{numsolve}

            $\points \gets \hyperref[alg:endpoints]{\mathsf{EndPoints}}(P(x), \tilde{h}_k, \tilde{h}_{k+1})$\;\label{callendp}

            \eIf{$\points \neq \varnothing$}{$\lceil h_k \rceil \gets \min(\points)$, and $\lfloor h_{k+1} \rfloor \gets \max(\points)$\;}{continue\;}
            
            \lFor{$\displaystyle l = (n-i)! \left\lceil \frac{\left\lceil h_k \right\rceil}{(n-i)!} \right\rceil$ \KwTo $\lfloor h_{k+1} \rfloor$ \KwBy $(n-i)!$}
            {append $P(x)+l$ to \temp} \label{rangeC}
            }
        $\out \gets \temp$, and $\temp \gets \varnothing$\;
        }
    }
    {\Return{\out}\;\label{arrend}}
    
    \Return{\out}\;   
\end{algorithm}

\section{Enumerating Seidel-feasible polynomials} \label{sec:Seidelcharalgo}

In this section, we describe algorithms that generate all \emph{Seidel-feasible polynomials}.
We will first define Seidel-feasible polynomials of even degree in Section~\ref{subsec:evendeg} and then define Seidel-feasible polynomials of odd degree separately in Section~\ref{subsec:oddeg}.

Let $M$ be a real symmetric matrix and we define the \textbf{characteristic polynomial} of $M$ as $\Char_M(x) \coloneqq \det(xI-M)$.
Let $S$ be a Seidel matrix of order $n \geqslant 1$.
Let $a_0=1$ and let $a_1, a_2, \dots, a_n$ be integers such that $\displaystyle \Char_S(x) = \sum_{i=0}^n a_i x^{n-i}$.
Since $\tr S = 0$ and $\tr S^2 = n(n-1)$, we obtain $a_1 = 0$ and $\displaystyle a_2 = -\binom{n}{2}$.
The coefficients of $\Char_S(x)$ also satisfy several divisibility and number-theoretic constraints \cite{GSY21, GY19}.
Adopting the terminology from \cite{GSY21}, we use the next definition to convey an important part of these constraints.

\begin{dfn} \label{type2}
Let $p(x) \in \mathbb{Z}[x]$ be a monic polynomial of degree $n \geqslant 0$.
Let $b_0 = 1$ and let $b_1, b_2, \dots, b_n$ be integers such that $\displaystyle p(x) = \sum_{i=0}^n b_i x^{n-i}$.
We call $p(x)$ a \emph{\textbf{type~2}} polynomial if $2^i$ divides $b_i$ for all $i \in \{1,2,\dots,n\}$ and a \emph{\textbf{weakly type~2}} polynomial if $2^{i-1}$ divides $b_i$ for all $i \in \{1,2,\dots,n\}$.
\end{dfn}

Clearly, any type~2 polynomial is also weakly type~2.
The following lemma is directly obtained from \cite[Lemma 2.7]{GSY21}.

\begin{lem} \label{lem:CharStype}
Let $S$ be a Seidel matrix of order $n \geqslant 1$.
Then $\Char_S(x-1)$ is weakly type~2.
Furthermore, if $n$ is even then $\Char_S(x-1)$ is type~2.
\end{lem}

Let $n \geqslant 1$ be an integer.
Let $p(x)$ be an integer polynomial of degree $n$ and let $a_0, a_1, \dots, a_n$ be integers such that $\displaystyle p(x) = \sum_{i=0}^n a_i x^{n-i}$.
We call $p(x)$ a \textbf{Seidel trace polynomial} if $p(x)$ is real-rooted, $a_0 = 1$, $a_1 = 0$, and $\displaystyle a_2 = -\binom{n}{2}$ if $n \geqslant 2$.
Let $\mathcal{T}_n$ denote the set of all Seidel trace polynomials of degree $n$.
For example, we have $\mathcal{T}_1 = \{x\}$ and $\mathcal{T}_2 = \{(x+1)(x-1)\}$.
We can use \hyperref[alg:numrealrooted]{$\mathsf{AllRealRooted}$} to enumerate $\mathcal{T}_n$ for small $n$.
Suppose that $p(x) \in \mathbb{R}[x]$ is a non-constant polynomial.
We call $p(x)$ \textbf{Seidel-realisable} if there exists a Seidel matrix $S$ such that $\Char_S(x) = p(x)$.
Let $\mathcal{R}_n$ denote the set of all Seidel-realisable polynomials of degree $n$.
It follows that $\mathcal{R}_n \subseteq \mathcal{T}_n$.
In particular, we have $\mathcal{R}_1 = \mathcal{T}_1$ and $\mathcal{R}_2 = \mathcal{T}_2$.
Let $p(x)$ be a Seidel-realisable polynomial of degree $n \geqslant 1$.
Applying Lemma~\ref{lem:CharStype}, we have that $p(x-1)$ is weakly type~2, and type~2 if $n$ is even.
Suppose that we have a monic real-rooted polynomial $q(x) \in \mathbb{Z}[x]$ that is also a divisor of $p(x)$.
As an example, we can use \cite[Lemma 3.1]{GSY21} to find such $q(x)$.
Then, we can write $p(x) = q(x) \cdot f(x)$ for some monic real-rooted polynomial $f(x) \in \mathbb{Z}[x]$.
Lemma~\ref{lem:typefactor} below implies that both $q(x-1)$ and $f(x-1)$ must be weakly type~2.
Moreover, if $n$ is even then $q(x-1)$ and $f(x-1)$ are both type~2.

\begin{lem}[{\cite[Lemma 2.8]{GSY21}}] \label{lem:typefactor}
    Let $p(x) \in \mathbb{Z}[x]$ be a monic polynomial.
    Suppose that $p(x)$ is equal to $q(x) \cdot r(x)$ for some monic integer polynomials $q(x)$ and $r(x)$.
    Then
    \begin{enumerate}[label=(\roman*)]
        \item $p(x)$ is type~2 if and only if $q(x)$ and $r(x)$ are both type~2.
        \item $p(x)$ is weakly type~2 if and only if $q(x)$ and $r(x)$ are both weakly type~2 and at least one of them is type~2.
    \end{enumerate}
\end{lem}

We will next implement algorithms to enumerate \emph{Seidel-feasible polynomials}.
The definition is motivated by key necessary conditions that are satisfied by any Seidel-realisable polynomial.
Conversely, we must have that any Seidel-realisable polynomial is Seidel-feasible.
We provide separate definitions for even and odd degrees, as the odd case requires a different approach.

\subsection{Seidel-feasible polynomials of even degree} 
\label{subsec:evendeg}

Fix a positive even integer $n$.
Let $p(x)$ be an integer polynomial of degree $n$.
We call $p(x)$ a \textbf{Seidel-feasible polynomial} if $p(x) \in \mathcal{T}_n$ and $p(x-1)$ is type~2.
We modify \hyperref[alg:numrealrooted]{$\mathsf{AllRealRooted}$} to implement \hyperref[alg:Seidelcharalgoeven]{$\mathsf{FeasibleEven}$} that enumerates Seidel-feasible polynomials of even degree.
The inputs of \hyperref[alg:Seidelcharalgoeven]{$\mathsf{FeasibleEven}$} are an even integer $n \geqslant 2$ together with a monic polynomial $q(x) \in \mathbb{Z}[x]$ of degree $s$ where $0 \leqslant s \leqslant n-2$.
The output will be all Seidel-feasible polynomials of degree $n$ that are divisible by $q(x)$.
Hence, we want to exhaustively find every integer polynomial $f(x)$ of degree $t = n-s$ such that $q(x) \cdot f(x)$ is a Seidel-feasible polynomial.
Note that, by definition, $q(x-1) \cdot f(x-1)$ is type~2.
Since $n$ is even, by Lemma~\ref{lem:typefactor}, we have that $q(x-1)$ and $f(x-1)$ are both type~2.
Hence, in Line~\ref{checkqx} we first check that if $q(x)$ is not real-rooted or $q(x-1)$ is not type~2, then the algorithm returns an empty array.
Let $d_0, d_1, \dots, d_s$ be integers such that $\displaystyle q(x) = \sum_{i=0}^s d_i x^{s-i}$ and let $c_0, c_1, \dots, c_t$ be integers such that $\displaystyle f(x) = \sum_{i=0}^t c_i x^{t-i}$.
It follows that $c_0 = d_0 =1$, $c_1 = -d_1$, and $\displaystyle c_2 = {d_1}^2-d_2-\binom{n}{2}$, where we set $d_1=0$ if $s=0$, and $d_2=0$ if $s \leqslant 1$.
Since $f(x-1)$ is type~2, \hyperref[alg:Seidelcharalgoeven]{$\mathsf{FeasibleEven}$} will indirectly enumerate all possible $f(x)$ by enumerating all possible $f(x-1)$ instead.
In Line~\ref{p2fhat} we let $p_2(x)$ be the $(t-2)$th derivative of $\hat{f}(x-1)$, where
\[
\hat{f}(x) = x^t - d_1 x^{t-1} + \left[ {d_1}^2-d_2-\binom{n}{2} \right] x^{t-2}.
\]
Line~\ref{startsimilar} to Line~\ref{endsimilar} are based on the steps in Line~\ref{arrstart} to Line~\ref{arrend} of \hyperref[alg:numrealrooted]{$\mathsf{AllRealRooted}$}.
Due to the type~2 condition, we have the factor $2^i$ in Line~\ref{pow2range}.
In Line~\ref{shiftback} and Line~\ref{timesqx}, we substitute $x$ with $x+1$ for each polynomial that we have obtained and then multiply it by $q(x)$.

\begin{algorithm}[htbp]
\capstart
\caption{$\mathsf{FeasibleEven}(n, q(x))$}
\label{alg:Seidelcharalgoeven}
\SetKwData{out}{out}
\SetKwData{temp}{temp}
\SetKwData{points}{points}
\SetKwInOut{Input}{Input}
\SetKwInOut{Output}{Output}
\SetKw{KwBy}{by}
\Input{An even integer $n \geqslant 2$, and a monic polynomial $\displaystyle q(x) = \sum_{i=0}^s d_i x^{s-i} \in \mathbb{Z}[x]$ where $s \in \{0,1,\dots,n-2\}$.}
\Output{All Seidel-feasible polynomials of degree $n$ that are divisible by $q(x)$.}

    $t \gets n-s$, and initialise each of \temp and \out as an empty array\;
    \lIf{\hyperref[alg:realrootedcheck]{$\mathsf{IsRealRooted}$}$(q(x)) = \False$ {\bf or} $q(x-1)$ is not type~2}{\Return{\out}}\label{checkqx}
    
    $\displaystyle \hat{f}(x) \gets x^t - d_1 x^{t-1} + \left[ {d_1}^2-d_2-\frac{n(n-1)}{2} \right] x^{t-2}$, where we set $d_1=0$ if $s=0$, and $d_2=0$ if $s \leqslant 1$\;
    $p_2(x) \gets$ $(t-2)$th derivative of $\hat{f}(x-1)$\;\label{p2fhat}
    \eIf{\hyperref[alg:realrootedcheck]{$\mathsf{IsRealRooted}$}$(p_2(x)) = \True$\label{startsimilar}}
    {
        Append $p_2(x)$ to \out\;
        \For{$i = 3$ \KwTo $t$}
        {
            \For{$j = 1$ \KwTo $\# \out$}
            {
                $\displaystyle k \gets \left\lfloor \frac{i-1}{2} \right\rfloor$, $p(x) \gets$ the $j$th entry of \out, and $\displaystyle P(x) \gets \int_0^x p(y) \dd{y}$\;
                
                Find numerical approximations $\tilde{h}_k$ and $\tilde{h}_{k+1}$ of $h_k$ and $h_{k+1}$, respectively, where $h_1 \leqslant h_2 \leqslant \dots \leqslant h_{i-1}$ are all roots of $\Disc_x(P(x)+C)$ as a univariate polynomial in $C$\;

                $\points \gets \hyperref[alg:endpoints]{\mathsf{EndPoints}}(P(x), \tilde{h}_k, \tilde{h}_{k+1})$\;

                \eIf{$\points \neq \varnothing$}{$\lceil h_k \rceil \gets \min(\points)$, and $\lfloor h_{k+1} \rfloor \gets \max(\points)$\;}{continue\;}
                
                \lFor{$\displaystyle l = 2^i (t-i)! \left\lceil \frac{\lceil h_k \rceil}{2^i (t-i)!} \right\rceil$ \KwTo $\lfloor h_{k+1} \rfloor$ \KwBy $2^i (t-i)!$}
                {append $P(x)+l$ to \temp}\label{pow2range}
            }
            $\out \gets \temp$, and $\temp \gets \varnothing$\;
        }
    }
    {\Return{\out}\;\label{endsimilar}}
    
    Update \out by substituting $x$ with $x+1$ for each polynomial in \out\;\label{shiftback}

    Multiply each polynomial in \out by $q(x)$\;\label{timesqx}
    
    \Return{\out}\;

\end{algorithm}

\subsection{Seidel-feasible polynomials of odd degree}
\label{subsec:oddeg}

First, we recall some properties of the coefficients of Seidel-realisable polynomials of odd degree that have been derived in \cite{GY19}.

Let $n \geqslant 1$ be an integer.
Let $I_n$ be the identity matrix of order $n$ and let $J_n$ be the all-ones matrix of order $n$.
If the order of the matrix is clear from the context, we simply write $I$ and $J$ instead.
Let $S$ be a Seidel matrix of order $n$.
Then there exists a graph with adjacency matrix $A$ such that $S = J-I-2A$.
Clearly, we also have that $\Char_{J-2A}(x) = \Char_{S}(x-1)$.

\begin{lem}[{\cite[Lemma 3.8]{GY19}}] \label{lem:evencoeff}
    Let $A$ be an adjacency matrix of a graph of order $n \geqslant 1$.
    Let $b_0, b_1, \dots, b_n$ be integers such that $\displaystyle \Char_{J-2A}(x) = \sum_{i=0}^n b_i x^{n-i}$.
    Then $2^i$ divides $b_i$ for all even $i$ in $\{4, 5, \dots, n\}$.
\end{lem}

If $n$ is a positive integer, let $\varphi(n)$ denote the Euler's totient function of $n$.
A graph $\Gamma$ is an \textbf{Euler graph} if each vertex of $\Gamma$ has even degree.
Here, we assign the value $1$ to $0^0$.

\begin{lem}[cf.~{\cite[Lemma 3.12]{GY19}}] \label{lem:modd}
    Let $\Gamma$ be an Euler graph of order $n$ odd, and let $A$ be its adjacency matrix.
    Let $b_0, b_1, \dots, b_n$ be integers such that $\displaystyle \Char_{J-2A}(x) = \sum_{i=0}^n b_i x^{n-i}$.
    Then, for all odd $i$ in $\{ 5, 6, \dots, n \}$, we have
    \begin{align} \label{eq:modd}
    b_i \equiv \sum_{d \, | \, i-1} \; \sum_{\substack{m_1+2m_2+\dots+dm_d=d \mathstrut \\ m_1 \geqslant 0, \dots, m_d \geqslant 0 \mathstrut \\ m_{i-1}=0}} C_d(m_1,\dots,m_d) \cdot P_d(m_1,\dots,m_d) \mod 2^i
    \end{align}
    where
    \begin{align*}
        C_d(m_1,\dots,m_d) &\coloneqq 2^{i-1} \cdot \frac{d \cdot \varphi\left( \frac{i-1}{d} \right)}{i-1} \cdot \frac{(m_1+m_2+\dots+m_d-1)!}{m_1! m_2! \cdots m_d!},\; \text{and} \\
        P_d(m_1,\dots,m_d) &\coloneqq \prod_{j=1}^{\left\lfloor \frac{d}{2} \right\rfloor} \left( \frac{b_{2j+1}}{2^{2j} \cdot n} \right)^{m_{2j}} \cdot \prod_{j=1}^{\left\lfloor \frac{d+1}{2} \right\rfloor} \left( \frac{b_{2j+1}+b_{2j}+b_{2j-1} \cdot b_3}{2^{2j-1}} \right)^{m_{2j-1}}.
    \end{align*}
\end{lem}

Let $n \geqslant 1$ be an odd integer and let $p(x)$ be an integer polynomial of degree $n$.
We call $p(x)$ a \textbf{Seidel-feasible polynomial} if $p(x) \in \mathcal{T}_n$, $p(x-1)$ is weakly type~2, and the next condition holds:
\begin{enumerate}[label=(\roman*), ref=\roman*]
    \item \label{i} If $b_0, b_1, \dots, b_n$ are integers such that $\displaystyle p(x-1) = \sum_{i=0}^n b_i x^{n-i}$, then $2^i$ divides $b_i$ for all even $i$ in $\{4, 5, \dots, n\}$, and $b_i$ satisfies \eqref{eq:modd} for all odd $i$ in $\{ 5, 6, \dots, n\}$.
\end{enumerate}
Condition~\eqref{i} is motivated by both Lemma~\ref{lem:evencoeff} and Lemma~\ref{lem:modd}.
We implement \hyperref[alg:checkmod]{$\mathsf{ModCheck}$} specifically for checking Condition~\eqref{i}.
Let $\bar{p}(x)$ be an integer polynomial of odd degree $n \geqslant 5$.
Let $b_0, b_1, \dots, b_n$ be integers such that $\displaystyle \bar{p}(x) = \sum_{k=0}^n b_k x^{n-k}$.
The algorithm \hyperref[alg:checkmod]{$\mathsf{ModCheck}$} takes as inputs $\bar{p}(x)$ together with an integer $i$ where $4 \leqslant i \leqslant n$.
It returns \texttt{True} if $2^i$ divides $b_i$ whenever $i$ is even, or if $b_i$ satisfies \eqref{eq:modd} whenever $i$ is odd.
Otherwise, \hyperref[alg:checkmod]{$\mathsf{ModCheck}$} will return \texttt{False}.

\begin{algorithm}[htbp]
\capstart
\caption{$\mathsf{ModCheck}(\bar{p}(x), i)$}
\label{alg:checkmod}
\SetKwData{out}{out}
\SetKwData{div}{div}
\SetKwData{total}{total}
\SetKwInOut{Input}{Input}
\SetKwInOut{Output}{Output}
\Input{An integer polynomial $\bar{p}(x) = \displaystyle \sum_{k=0}^n b_k x^{n-k}$ of odd degree $n \geqslant 5$, and an integer $i$ where $4 \leqslant i \leqslant n$.}
\Output{\texttt{True} if $2^i$ divides $b_i$ whenever $i$ is even, or if $b_i$ satisfies \eqref{eq:modd} whenever $i$ is odd.
Otherwise, the algorithm will return \False.}

$\out \gets \False$\;
$\total \gets 0$\;
\eIf{$i \equiv 0 \mod 2$}
    {
        \lIf{$b_i \equiv 0 \mod 2^i$}{$\out \gets \True$}
    }
    {
        Initialise \div as an array consisting of all positive divisors $d$ of $i-1$\;
        \For{each $d$ in \div}
        {
            Find all partitions of $d$\;
            \For{each partition of $d$}
            {
                Determine the corresponding $m_1, m_2, \dots, m_d$\;
                \lIf{$d = i-1$}{$m_d = 0$}
                Compute $C_d(m_1, \dots, m_d)$ and $P_d(m_1, \dots, m_d)$\;
                $\total \gets \total + C_d(m_1, \dots, m_d) \cdot P_d(m_1, \dots, m_d)$\;
            }
        }
        \lIf{$b_i \equiv \total \mod 2^i$}{$\out \gets \True$}
    }

\Return{\out}\;

\end{algorithm}

For computational purposes, we introduce the next definition.
Let $n \geqslant 5$ be an odd integer.
Let $p(x)$ be an integer polynomial of degree $n$ and let $\delta$ be an integer such that $4 \leqslant \delta \leqslant n$.
We call $p(x)$ a \textbf{$\delta$-partial Seidel-feasible polynomial} if $p(x) \in \mathcal{T}_n$, $p(x-1)$ is weakly type~2, and the following condition holds:
\begin{enumerate}[resume*]
    \item \label{ii} If $b_0, b_1, \dots, b_n$ are integers such that $\displaystyle p(x-1) = \sum_{i=0}^n b_i x^{n-i}$, then $2^i$ divides $b_i$ for all even $i$ in $\{4, 5, \dots, \delta\}$, and $b_i$ satisfies \eqref{eq:modd} for all odd $i$ in $\{ 5, 6, \dots, \delta\}$.
\end{enumerate}

We can now explain the steps in \hyperref[alg:Seidelcharalgodd]{$\mathsf{FeasiblePartial}$}.
The inputs of \hyperref[alg:Seidelcharalgodd]{$\mathsf{FeasiblePartial}$} are an odd integer $n \geqslant 5$, a monic polynomial $q(x) \in \mathbb{Z}[x]$ of degree $s$ where $0 \leqslant s \leqslant n-2$, and an integer $\delta$ where $\max(4, n-s) \leqslant \delta \leqslant n$.
The output will be all $\delta$-partial Seidel-feasible polynomials of degree $n$ that are divisible by $q(x)$.
\hyperref[alg:Seidelcharalgodd]{$\mathsf{FeasiblePartial}$} starts in exactly the same way as \hyperref[alg:Seidelcharalgoeven]{$\mathsf{FeasibleEven}$}.
However, beginning at Line~\ref{branchout}, \hyperref[alg:Seidelcharalgodd]{$\mathsf{FeasiblePartial}$} will do a different search based on the definition of $\delta$-partial Seidel-feasible polynomial.
Let $f(x)$ be an integer polynomial of degree $t=n-s \geqslant 2$ such that $q(x) \cdot f(x)$ is a $\delta$-partial Seidel-feasible polynomial.
By Lemma~\ref{lem:typefactor}, we have that $f(x-1)$ is weakly type~2.
Let $\bar{c}_0, \bar{c}_1, \dots, \bar{c}_t$ be integers such that $\displaystyle f(x-1) = \sum_{i=0}^t \bar{c}_i x^{t-i}$.
For every $i \in \{2, 3, \dots, t\}$, we let $\displaystyle \bar{f}_i(x) = \dv[t-i]{x} f(x-1)$.
We have the following relation
\begin{equation} \label{eq:consecdervfbar}
    \bar{f}_i(x) = \int_0^x \bar{f}_{i-1}(y) \dd{y} + (t-i)! \cdot \bar{c}_i
\end{equation}
for all $i \in \{3, 4, \dots, t\}$.
Suppose that we have determined $\bar{f}_{i-1}(x)$ for some $3 \leqslant i \leqslant t$.
From $\bar{f}_{i-1}(x)$, we want to enumerate all possible $\bar{f}_i(x)$.
Let $\displaystyle P(x) = \int_0^x \bar{f}_{i-1}(y) \dd{y}$ so that $\bar{f}_i(x) = P(x) + C$ for some $C \in \mathbb{Z}$.
First, we consider the case where $i=3$.
We know that $f(x-1)$ is weakly type~2 so $\bar{c}_i$ is divisible by $4$.
By $\eqref{eq:consecdervfbar}$, we have that $C$ must be divisible by $4(t-i)!$.
Hence, the possible values of $C$ are all integers from $\lceil h_k \rceil$ to $\lfloor h_{k+1} \rfloor$ that are also divisible by $4(t-i)!$.
We implement this in Line~\ref{i3} and Line~\ref{allC} of \hyperref[alg:Seidelcharalgodd]{$\mathsf{FeasiblePartial}$}.
Otherwise, suppose that $i>3$, which we cover in Line~\ref{i>3start} to Line~\ref{i>3end}.
Since $f(x-1)$ is weakly type~2, we have that $\bar{c}_i$ is divisible by $2^{i-1}$.
Condition~\eqref{ii} fixes a single value for $b_i$ modulo $2^i$.
Examining the product of $q(x-1)$ and $f(x-1)$, we observe that $\bar{c}_i$ can only satisfy exactly one of either $\bar{c}_i \equiv 0 \mod 2^i$ or $\bar{c}_i \equiv 2^{i-1} \mod 2^i$.
Thus, the least possible value of $C$ must be equal to either one of
\begin{equation*}
    C_1 = 2^{i-1} (t-i)! \left\lceil \frac{\lceil h_k \rceil}{2^{i-1} (t-i)!} \right\rceil, \; \text{or} \;\, C_2 = C_1 + 2^{i-1} (t-i)!.
\end{equation*}
Note that $C_1$ is the least integer greater than or equal to $\lceil h_k \rceil$ that is divisible by $2^{i-1} (t-i)!$.
To generate all possible values of $C$, starting from $C_1$ or $C_2$, we iteratively add $2^i (t-i)!$ while ensuring that the value does not exceed $\lfloor h_{k+1} \rfloor$.
We decide between $C_1$ or $C_2$ by using \hyperref[alg:checkmod]{$\mathsf{ModCheck}$}.
The polynomial $\bar{f}_{i-1}(x)$ determines the values of the coefficients $\bar{c}_0, \bar{c}_1, \dots, \bar{c}_{i-1}$ of $f(x-1)$.
We let 
\[
\bar{P}_1(x) = \sum_{j=0}^{i-1} \bar{c}_j x^{t-j} + \frac{C_1}{(t-i)!} x^{t-i}, \; \text{and} \;\, \bar{P}_2(x) = \sum_{j=0}^{i-1} \bar{c}_j x^{t-j} + \frac{C_2}{(t-i)!} x^{t-i}.
\]
If \hyperref[alg:checkmod]{$\mathsf{ModCheck}$}$\left( q(x-1) \cdot \bar{P}_1(x), i \right)$ returns \texttt{True}, then we choose $C_1$.
Otherwise, if it returns \texttt{False}, then we run \hyperref[alg:checkmod]{$\mathsf{ModCheck}$}$\left( q(x-1) \cdot \bar{P}_2(x), i \right)$.
If it returns \texttt{True}, then we choose $C_2$.
If both return \texttt{False}, then we can continue directly to the next possible $\bar{f}_{i-1}(x)$.
The corresponding steps in \hyperref[alg:Seidelcharalgodd]{$\mathsf{FeasiblePartial}$} are from Line~\ref{Pbar} to Line~\ref{i>3end}.
After we enumerate all possible $f(x-1)$, in Line~\ref{qxshift} we multiply each one of them by $q(x-1)$.
For any candidate $\bar{r}(x) = q(x-1) \cdot f(x-1)$, note that we only check Condition~\eqref{ii} for all $4 \leqslant i \leqslant t$.
Therefore, in Line~\ref{rbarfilter} we remove any candidate $\bar{r}(x)$ if there exists an integer $i$ such that $\max(4, t+1) \leqslant i \leqslant \delta$ and $\hyperref[alg:checkmod]{\mathsf{ModCheck}}(\bar{r}(x), i)$ returns \texttt{False}.
We also remove any $\bar{r}(x)$ that is not weakly type~2.
Lastly, we substitute $x$ with $x+1$ for the remaining polynomials.

Let $n \geqslant 1$ be an integer.
Denote by $\mathcal{F}_n$ the set of all Seidel-feasible polynomials of degree $n$.
We have that $\mathcal{R}_n \subseteq \mathcal{F}_n \subseteq \mathcal{T}_n$.

\begin{remark} \label{rem:deg5n6}
    We find that $\mathcal{R}_n = \mathcal{F}_n$ for all integers $1 \leqslant n \leqslant 5$.
    For degree 6, there already exist polynomials $p(x)$ such that $p(x) \in \mathcal{F}_6$ but $p(x) \notin \mathcal{R}_6$.
    One such example is the polynomial $p(x) = (x+1)^2(x-1)(x^3-x^2-13x+5)$ in Example~\ref{ex:interlacingseidelchi} in the next section.
\end{remark}

Suppose that $n \geqslant 5$ is odd and let $\delta$ be an integer such that $4 \leqslant \delta \leqslant n$.
Denote by $\mathcal{F}_n(\delta)$ the set of all $\delta$-partial Seidel-feasible polynomials of degree $n$.
Observe that
\[
\mathcal{R}_n \subseteq \mathcal{F}_n = \mathcal{F}_n(n) \subseteq \mathcal{F}_n(n-1) \subseteq \dots \subseteq \mathcal{F}_n(4) \subseteq \mathcal{T}_n.
\]
One can use \hyperref[alg:Seidelcharalgodd]{$\mathsf{FeasiblePartial}$} (or \hyperref[alg:Seidelcharalgoeven]{$\mathsf{FeasibleEven}$} if $n$ is even) to enumerate candidate characteristic polynomials for a Seidel matrix corresponding to a system of $n$ equiangular lines in $\mathbb R^d$.
Typically, one assumes that $n > 2d$.
In this case, by a theorem attributed to Neumann~\cite[p. 498]{LS73}, the smallest root of each candidate characteristic polynomial is an odd integer $\lambda$ of multiplicity at least $n-d$.
Accordingly, we want to find all polynomials in $\mathcal F_n$ that are divisible by $(x-\lambda)^{n-d}$.
Based on empirical evidence, it suffices to do the enumeration in $\mathcal F_n(\delta)$ where $\delta$ is much smaller than $n$.
This motivates the following question.

\begin{qn} \label{qn:delta}
    Fix an integer $d\geqslant 2$.
    Let $n \geqslant 2d+1$ be an odd integer and let $\delta$ be an integer such that $4 \leqslant \delta \leqslant n$.
    Let $\lambda$ be an odd negative integer and let $\mathcal G_n(\delta, d, \lambda)$ be the set of all polynomials in $\mathcal{F}_n(\delta)$ that are divisible by $(x-\lambda)^{n-d}$.
    Let $\delta(n, d, \lambda)$ be the minimum value of $\delta$ such that $\mathcal G_n(\delta, d, \lambda) = \mathcal G_n(n, d, \lambda)$.
    What is the behavior of $\delta(n, d, \lambda)$ relative to n?
\end{qn}

As an example, if $n=23$, $d=7$, and $\lambda=-3$, then $\delta(23, 7, -3) = 9$.

\begin{algorithm}[htbp]
\capstart
\caption{$\mathsf{FeasiblePartial}(n, q(x), \delta)$}
\label{alg:Seidelcharalgodd}
\SetKwData{out}{out}
\SetKwData{temp}{temp}
\SetKwData{points}{points}
\SetKwData{start}{start}
\SetKwData{period}{period}
\SetKwInOut{Input}{Input}
\SetKwInOut{Output}{Output}
\SetKw{KwBy}{by}
\Input{An odd integer $n \geqslant 5$, a monic polynomial $\displaystyle q(x) = \sum_{i=0}^s d_i x^{s-i} \in \mathbb{Z}[x]$ where $0 \leqslant s \leqslant n-2$, and an integer $\delta$ where $\max(4, n-s) \leqslant \delta \leqslant n$.}
\Output{All $\delta$-partial Seidel-feasible polynomials of degree $n$ that are divisible by $q(x)$.}

    $t \gets n-s$, and initialise each of \temp and \out as an empty array\;
    \lIf{\hyperref[alg:realrootedcheck]{$\mathsf{IsRealRooted}$}$(q(x)) = \False$ {\bf or} $q(x-1)$ is not weakly type~2}{\Return{\out}}
    
    $\displaystyle \hat{f}(x) \gets x^t - d_1 x^{t-1} + \left[ {d_1}^2-d_2-\frac{n(n-1)}{2} \right] x^{t-2}$, where we set $d_1=0$ if $s=0$, and $d_2=0$ if $s \leqslant 1$\;
    $p_2(x) \gets$ $(t-2)$th derivative of $\hat{f}(x-1)$\;
    \eIf{\hyperref[alg:realrootedcheck]{$\mathsf{IsRealRooted}$}$(p_2(x)) = \True$}
    {
        Append $p_2(x)$ to \out\;
        \For{$i = 3$ \KwTo $t$}
        {
            \For{$j = 1$ \KwTo $\# \out$}
            {
                $\displaystyle k \gets \left\lfloor \frac{i-1}{2} \right\rfloor$, $p(x) \gets$ the $j$th entry of \out, and $\displaystyle P(x) \gets \int_0^x p(y) \dd{y}$\;
                
                Find numerical approximations $\tilde{h}_k$ and $\tilde{h}_{k+1}$ of $h_k$ and $h_{k+1}$, respectively, where $h_1 \leqslant h_2 \leqslant \dots \leqslant h_{i-1}$ are all roots of $\Disc_x(P(x)+C)$ as a univariate polynomial in $C$\;

                $\points \gets \hyperref[alg:endpoints]{\mathsf{EndPoints}}(P(x), \tilde{h}_k, \tilde{h}_{k+1})$\;

                \eIf{$\points \neq \varnothing$}{$\lceil h_k \rceil \gets \min(\points)$, and $\lfloor h_{k+1} \rfloor \gets \max(\points)$\;}{continue\;}
                
                \eIf{$i=3$\label{branchout}}
                {
                    $\period \gets 4(t-i)!$, and $\displaystyle \start \gets \period \left\lceil \frac{\lceil h_k \rceil}{\period} \right\rceil$\;\label{i3}
                }
                {
                    $\period \gets 2^i (t-i)!$, $\displaystyle C_1 \gets \frac{\period}{2} \left\lceil \frac{2\lceil h_k \rceil}{\period} \right\rceil$, and $\displaystyle C_2 \gets C_1+\frac{\period}{2}$\;\label{i>3start}
                    $\bar{P}_1(x) \gets P(x)+C_1$, and $\bar{P}_2(x) \gets P(x)+C_2$\;\label{Pbar}
                    \lFor{$m=1$ \KwTo $t-i$}{$\displaystyle \bar{P}_1(x) \gets \int_0^x \bar{P}_1(y) \dd{y}$, and $\displaystyle \bar{P}_2(x) \gets \int_0^x \bar{P}_2(y) \dd{y}$}
                    \uIf{$\hyperref[alg:checkmod]{\mathsf{ModCheck}}$\(\left( q(x-1) \cdot \bar{P}_1(x), i \right) = \True\)}{$\start \gets C_1$\;}
                    \uElseIf{$\hyperref[alg:checkmod]{\mathsf{ModCheck}}$\(\left( q(x-1) \cdot \bar{P}_2(x), i \right) = \True\)}{$\start \gets C_2$\;}
                    \Else{continue\;\label{i>3end}}
                }
                
                \lFor{$l = \start$ \KwTo $\lfloor h_{k+1} \rfloor$ \KwBy \period}
                {append $P(x)+l$ to \temp}\label{allC}
                    
            }
            $\out \gets \temp$, and $\temp \gets \varnothing$\;
        }
    }
    {\Return{\out}\;}

    Multiply each polynomial in \out by $q(x-1)$\;\label{qxshift}

    Remove any polynomial $\bar{r}(x)$ in \out if $\bar{r}(x)$ is not weakly type~2, or if there exists an integer $i$ such that $\max(4, t+1) \leqslant i \leqslant \delta$ and $\hyperref[alg:checkmod]{\mathsf{ModCheck}}(\bar{r}(x), i) = \False$\;\label{rbarfilter}

    Update \out by substituting $x$ with $x+1$ for each remaining polynomial in \out\;
    
    \Return{\out}\;

\end{algorithm}

\section{Enumerating Seidel interlacing polynomials utilising linear programming} \label{sec:LPalgo}

In this section, we describe linear programming algorithms that produce all \emph{($\delta$-partial) Seidel interlacing polynomials} of a given Seidel trace polynomial.

Let $n \geqslant 2$ be an integer.
Let $p(x), q(x) \in \mathbb{R}[x]$ be real-rooted polynomials such that $\deg p(x) = n$ and $\deg q(x) = n-1$.
Let $\mu_1 \leqslant \mu_2 \leqslant \dots \leqslant \mu_{n-1}$ be the roots of $q(x)$ and let $\lambda_1 \leqslant \lambda_2 \leqslant \dots \leqslant \lambda_n$ be the roots of $p(x)$.
We say that $q(x)$ \textbf{interlaces} $p(x)$ if $\lambda_i \leqslant \mu_i \leqslant \lambda_{i+1}$ for all $i \in \{1, 2, \dots, n-1\}$.
Let $M$ be a real symmetric matrix of order $n \geqslant 2$.
Let $i \in \{1,2,\dots,n\}$ and denote by $M[i]$ the principal submatrix of $M$ obtained by deleting its $i$th row and column.
The Cauchy's interlacing theorem states that $\Char_{M[i]}(x)$ interlaces $\Char_M(x)$.

\begin{lem}[\cite{Fisk05, Hwang04}] \label{lem:cauchyinterlacing}
    Let $M$ be a real symmetric matrix of order $n \geqslant 2$ and let $i \in \{1, 2, \dots, n\}$.
    Then $\Char_{M[i]}(x)$ interlaces $\Char_M(x)$.
\end{lem}

Let $S$ be a Seidel matrix of order $n \geqslant 2$.
By Lemma~\ref{lem:cauchyinterlacing}, for all $i \in \{1, 2, \dots, n\}$, we have that $\Char_{S[i]}(x)$ interlaces $\Char_S(x)$.

Let $n \geqslant 2$ be an integer.
Let $p(x) \in \mathcal{T}_n$ and let $q(x)$ be an integer polynomial of degree $n-1$.
Denote by $\mathcal{J}(p(x))$ the set of all $q(x)$ such that $q(x) \in \mathcal{T}_{n-1}$ and $q(x)$ interlaces $p(x)$.
Furthermore, if $q(x) \in \mathcal{F}_{n-1}$ and $q(x)$ interlaces $p(x)$, then we call $q(x)$ a \textbf{Seidel interlacing polynomial} of $p(x)$.
Denote by $\mathcal{I}(p(x))$ the set of all Seidel interlacing polynomials of $p(x)$.
Clearly, we have $\mathcal{I}(p(x)) \subseteq \mathcal{J}(p(x))$.
Suppose that $n \geqslant 6$ is even and let $\delta$ be an integer such that $4 \leqslant \delta \leqslant n-1$.
If $q(x) \in \mathcal{F}_{n-1}(\delta)$ and $q(x)$ interlaces $p(x)$, then we call $q(x)$ a \textbf{$\delta$-partial Seidel interlacing polynomial} of $p(x)$.
Denote by $\mathcal{I}(p(x), \delta)$ the set of all $\delta$-partial Seidel interlacing polynomials of $p(x)$.
Observe that
\[
\mathcal{I}(p(x)) = \mathcal{I}(p(x), n-1) \subseteq \mathcal{I}(p(x), n-2) \subseteq \dots \subseteq \mathcal{I}(p(x), 4) \subseteq \mathcal{J}(p(x)).
\]

\begin{prop} \label{prop:interlacingchifactor}
    Let $n \geqslant 2$ be an integer.
    Let $p(x) \in \mathcal{T}_n$ and let $q(x) \in \mathcal{J}(p(x))$.
    Then $q(x) = \gcd(p(x), p^\prime(x)) \cdot f(x)$ for some monic real-rooted integer polynomial $f(x)$ that interlaces $\Min(p, x)$.
\end{prop}

\begin{proof}
    By \eqref{eq:minpx}, we have $p(x)=\gcd(p(x), p^\prime(x)) \cdot \Min(p, x)$.
    Since $q(x)$ interlaces $p(x)$, then there exists a monic real-rooted integer polynomial $f(x)$ such that $q(x) = \gcd(p(x), p^\prime(x)) \cdot f(x)$ and $f(x)$ interlaces $\Min(p, x)$.
\end{proof}

Let $p(x) \in \mathcal{T}_n$ and let $q(x) \in \mathcal{J}(p(x))$ for some integer $n \geqslant 2$.
Let $\displaystyle f(x) = \frac{q(x)}{\gcd(p(x), p^\prime(x))}$, which is an integer polynomial by Proposition~\ref{prop:interlacingchifactor}.
Analogous to \cite[Lemma 5.4]{GY19}, we can determine the first three leading coefficients of $f(x)$ from the coefficients of $\Min(p, x)$ as follows:

\begin{lem} \label{lem:interlacingfcoeff}
    Let $n \geqslant 3$ be an integer.
    Let $p(x) \in \mathcal{T}_n$ and let $r = \deg \Min(p, x)$ where $3 \leqslant r \leqslant n$.
    Let $q(x) \in \mathcal{J}(p(x))$ and let 
    \[
    f(x) = \frac{q(x)}{\gcd(p(x), p^\prime(x))} \in \mathbb{Z}[x].
    \]
    Let $b_0, b_1, \dots, b_r$ be integers such that $\displaystyle \Min(p, x) = \sum_{i=0}^r b_i x^{r-i}$ and let $c_0, c_1, \dots, c_{r-1}$ be integers such that $\displaystyle f(x) = \sum_{i=0}^{r-1} c_i x^{r-1-i}$.
    Then, we have $c_0 = b_0 = 1$, $c_1 = b_1$, and $c_2 = b_2+n-1$.
\end{lem}

\begin{proof}
    Note that $p(x) = \gcd(p(x), p^\prime(x)) \cdot \Min(p, x)$ and $q(x) = \gcd(p(x), p^\prime(x)) \cdot f(x)$.
    Let $d_0$, $d_1$, $\dots$, $d_{n-r}$ be integers such that $\displaystyle \gcd(p(x), p^\prime(x)) = \sum_{i=0}^{n-r} d_i x^{n-r-i}$.
    Since $p(x)$ and $q(x)$ are both monic, we have $d_0 = c_0 = b_0 = 1$.
    We deduce that
    \begin{align*}
        b_1 + d_1 = c_1 + d_1 &= 0, \; b_2 + d_2 + b_1 d_1 = -\binom{n}{2}, \; \text{and} \;\, c_2 + d_2 + c_1 d_1 = -\binom{n-1}{2},
    \end{align*}
    where we set $d_1=0$ if $n=r$, and $d_2=0$ if $n-r \leqslant 1$.
    Therefore, we conclude that $c_1=b_1$ and $c_2=b_2+n-1$.
\end{proof}

Let $n \geqslant 0$ be an integer and let $p(x) \in \mathbb{R}[x]$ be a nonzero polynomial of degree $n$.
Let $\coeff(p)$ be a vector in $\mathbb{R}^{n+1}$ such that for all $i \in \{1,2,\dots,n+1\}$, the $i$th entry of $\coeff(p)$ is equal to the coefficient of $x^{n-i+1}$ in $p(x)$.
We call $\coeff(p)$ the \textbf{coefficient vector} of $p(x)$.
Let $m$ be a positive integer and let $\mathcal{P} = \{p_1(x), p_2(x), \dots, p_m(x)\}$ be a set of polynomials in $\mathbb{R}[x]$, each of degree $n$.
We define $\coeff(\mathcal{P}) \coloneqq \{\coeff(p_i) : 1 \leqslant i \leqslant m\}$.
In particular, we have that $\coeff(\{p(x)\}) = \{\coeff(p)\}$.

Let $p(x) \in \mathbb{R}[x]$ be a real-rooted polynomial of degree $n \geqslant 1$.
Let $r$ be the degree of $\Min(p, x)$ where $1 \leqslant r \leqslant n$.
Let $\lambda_1 < \lambda_2 < \dots < \lambda_r$ be distinct real numbers such that \( \displaystyle \Min(p, x) = \prod_{i=1}^r (x-\lambda_i) \).
For each $j \in \{1, 2, \dots, r\}$, define
\[ \displaystyle \Min_j(p,x) \coloneqq \prod_{\substack{1 \leqslant i \leqslant r, \mathstrut \\ i \neq j}} (x-\lambda_i). \]
Clearly, for all $j \in \{1, 2, \dots, r\}$, we have $\Min(p, x) = (x-\lambda_j) \cdot \Min_j(p, x)$.

\begin{thm} \label{thm:gamma}
    Let $n \geqslant 2$ be an integer.
    Let $p(x) \in \mathcal{T}_n$ and let $r = \deg \Min(p, x)$ where $1 \leqslant r \leqslant n$.
    Let $q(x) \in \mathcal{J}(p(x))$ and let 
    \[
    f(x) = \frac{q(x)}{\gcd(p(x), p^\prime(x))} \in \mathbb{Z}[x].
    \]
    Then, there exist nonnegative real numbers $\gamma_1,\gamma_2,\dots,\gamma_r$ such that
    \begin{equation*}
        \coeff(f(x-1)) = \sum_{j=1}^r \gamma_j \cdot \coeff(\Min_j(p, x-1)).
    \end{equation*}
\end{thm}

\begin{proof}
    By Proposition~\ref{prop:interlacingchifactor}, we have that $f(x-1)$ is real-rooted and it also interlaces $\Min(p, x-1)$.
    By \cite[Remark 1.21]{Fisk06}, which extends \cite[Lemma 1.20]{Fisk06}, there exist nonnegative real numbers $\gamma_1,\gamma_2,\dots,\gamma_r$ such that $\displaystyle \coeff(f(x-1)) = \sum_{j=1}^r \gamma_j \cdot \coeff(\Min_j(p, x-1))$.
\end{proof}

Theorem~\ref{thm:gamma} above allows us to leverage linear programming to efficiently enumerate Seidel interlacing polynomials.
This incorporation of linear programming eliminates the need for real-rooted integer polynomial enumeration steps from \hyperref[alg:numrealrooted]{$\mathsf{AllRealRooted}$}, resulting in a significant speedup of computations.
On the other hand, however, linear programming involves extensive use of numerical analysis, which can introduce numerical approximation errors.
One can get around these issues by using Lemma~\ref{lem:omegafloorceil}.

The input of \hyperref[alg:interlacingLPoddeven]{$\mathsf{InterlacingEven}$} is a Seidel trace polynomial $p(x)$ of odd degree $n \geqslant 3$ where $3 \leqslant \deg \Min(p, x) \leqslant n$.
The output is the set $\mathcal{I}(p(x))$.
Let $r$ be the degree of $\Min(p, x)$ so $3 \leqslant r \leqslant n$.
By Proposition~\ref{prop:interlacingchifactor}, we want to find all possible polynomials $f(x)$ of degree $r-1$ such that $\gcd(p(x), p^\prime(x)) \cdot f(x) \in \mathcal{I}(p(x))$.
Let $b_0, b_1, \dots, b_r$ be integers such that $\displaystyle \Min(p, x) = \sum_{i=0}^r b_i x^{r-i}$, and let $c_0, c_1, \dots, c_{r-1}$ be integers such that $\displaystyle f(x) = \sum_{i=0}^{r-1} c_i x^{r-1-i}$.
By Lemma~\ref{lem:interlacingfcoeff}, we have that $c_0 = b_0 = 1$, $c_1 = b_1$, and $c_2 = b_2+n-1$.
Similar to algorithms in Section~\ref{sec:Seidelcharalgo}, we enumerate all possible $f(x)$ indirectly by enumerating all possible $f(x-1)$.
Let $\bar{c}_0, \bar{c}_1, \dots, \bar{c}_{r-1}$ be integers such that $\displaystyle f(x-1) = \sum_{i=0}^{r-1} \bar{c}_i x^{r-1-i}$.
We obtain $\bar{c}_0 = 1$, $\bar{c}_1 = b_1-r+1$, and
\[
\bar{c}_2 = c_2 + \binom{r-1}{2} - c_1(r-2) = b_2+n-1+\frac{1}{2}(r-1-2b_1)(r-2).
\]
We store this information as the vector in Line~\ref{first3coeffs} of \hyperref[alg:interlacingLPoddeven]{$\mathsf{InterlacingEven}$}.
By Theorem~\ref{thm:gamma}, there exist nonnegative real numbers $\gamma_1,\gamma_2,\dots,\gamma_r$ such that
\begin{equation} \label{eq:cbargamma}
    \left( \bar{c}_0, \bar{c}_1, \dots, \bar{c}_{r-1} \right)^{\intercal} = \sum_{j=1}^r \gamma_j \cdot \coeff(\Min_j(p, x-1)).
\end{equation}
Let $\bm{\gamma} = (\gamma_1, \gamma_2, \dots, \gamma_r)^{\intercal}$.
Since the values of $\bar{c}_0$, $\bar{c}_1$, $\bar{c}_2$, and all coefficients of $\Min_j(p, x-1)$ are fixed, we can apply linear programming using \eqref{eq:cbargamma} to recursively determine bounds for $\bar{c}_i$ where $3 \leqslant i \leqslant r-1$.
As detailed in Line~\ref{numB}, we store all coefficients of $\Min_j(p, x-1)$ in a square matrix $B$ of order $r$.
However, note that $\Min_j(p, x-1)$ may not always be an integer polynomial.
It is far more efficient, especially with linear programming, to use numerical approximations instead of exact values for the entries of $B$.\footnote{In our computations, we use 75 significant digits of precision.}
Let $\bar{\bm{\gamma}} = B \bm{\gamma}$ (Line~\ref{gammabar}) and for all $i \in \{1,2,\dots,r\}$, denote by $\bar{\bm{\gamma}}(i)$ the $i$th entry of $\bar{\bm{\gamma}}$.
By \eqref{eq:cbargamma}, it follows that $\bar{\bm{\gamma}}(i) = \bar{c}_{i-1}$ for all $i \in \{1,2,\dots,r\}$.
Suppose that we have determined the vector $\left( \bar{c}_0, \bar{c}_1, \dots, \bar{c}_{i-1} \right)^{\intercal}$ for some $3 \leqslant i \leqslant r-1$.
We formulate our optimisation problems for $\bar{c}_i$ as the following linear programs:
\begin{align}
    &\min_{1 \leqslant k \leqslant i} \left\{ \bar{\bm{\gamma}}(i+1) : \bar{\bm{\gamma}}(k)=\bar{c}_{k-1}, \bm{\gamma} \geqslant \bm{0} \right\} \tag{$\mathrm{LP}_1$}\label{eq:LP1}, \\
    &\max_{1 \leqslant k \leqslant i} \left\{ \bar{\bm{\gamma}}(i+1) : \bar{\bm{\gamma}}(k)=\bar{c}_{k-1}, \bm{\gamma} \geqslant \bm{0} \right\}. \tag{$\mathrm{LP}_2$}\label{eq:LP2}
\end{align}
Note that \eqref{eq:LP1} and \eqref{eq:LP2} are bounded since $0 \leqslant \gamma_j \leqslant 1$ for all $j \in \{1,2,\dots,r\}$.
We implement \eqref{eq:LP1} and \eqref{eq:LP2} in Line~\ref{LPmin} and Line~\ref{LPmax}, respectively.
Suppose that solving \eqref{eq:LP1} and \eqref{eq:LP2} yields us the optimum values $\mathsf{minLP}$ and $\mathsf{maxLP}$, respectively.
Since $\bar{c}_i$ is an integer, we are looking for the values of $\lceil \mathsf{minLP} \rceil$ and $\lfloor \mathsf{maxLP} \rfloor$.
Additionally, by Lemma~\ref{lem:typefactor}, we know that $f(x-1)$ is type~2 so $\bar{c}_i$ is divisible by $2^i$.
By Lemma~\ref{lem:omegafloorceil}, assuming that the absolute error is far less than $1/2$, the interval of integers due to $\mathsf{Round}(\mathsf{minLP})$ and $\mathsf{Round}(\mathsf{maxLP})$ will always contain the true exact interval of integers.
If $\mathsf{Round}(\mathsf{minLP})$ or $\mathsf{Round}(\mathsf{maxLP})$ introduces redundant integers, then in the next iteration of $i$, the corresponding regions of \eqref{eq:LP1} or \eqref{eq:LP2} will be infeasible.
Subsequently, we include Line~\ref{ilast} to Line~\ref{infeasible} specifically to eliminate these redundancies for the last iteration $i=r-1$.
At the end of \hyperref[alg:interlacingLPoddeven]{$\mathsf{InterlacingEven}$}, we substitute $x$ with $x+1$ to obtain all possible $f(x)$, and then multiply each one by $\gcd(p(x), p^\prime(x))$.

\begin{algorithm}[htbp]
\capstart
\caption{$\mathsf{InterlacingEven}(p(x))$}
\label{alg:interlacingLPoddeven}
\SetKwData{out}{out}
\SetKwData{temp}{temp}
\SetKwData{fixed}{fixed}
\SetKwData{constraints}{constraints}
\SetKwData{vec}{vec}
\SetKwData{minlp}{minLP}
\SetKwData{maxlp}{maxLP}
\SetKwData{start}{start}
\SetKwData{en}{end}
\SetKw{KwBy}{by}
\SetKwFunction{True}{True}
\SetKwInOut{Input}{Input}\SetKwInOut{Output}{Output}
\Input{A Seidel trace polynomial $p(x)$ of odd degree $n \geqslant 3$ where $3 \leqslant \deg \Min(p, x) \leqslant n$.}
\Output{The set $\mathcal{I}(p(x))$ of all Seidel interlacing polynomials of $p(x)$.}

$r \gets \deg \Min(p, x)$\;

Let $\displaystyle \Min(p, x) = \sum_{i=0}^r b_i x^{r-i}$ for some integers $b_0, b_1, \dots, b_r$\;

Initialise each of \temp and \out as an empty array\;

Let $B$ be a square matrix of order $r$ such that for all $i, j \in \{1, 2, \dots, r\}$, we have that $B_{i, j}$ is equal to a numerical approximation of the $i$th entry of $\coeff(\Min_j(p, x-1))$ with a suitably high digit precision\;\label{numB}

Initialise $\bm{\gamma}$ as a vector of variables $\gamma_1, \gamma_2, \dots, \gamma_r$, and $\bar{\bm{\gamma}} \gets B \bm{\gamma}$\;\label{gammabar}

Append the vector $\mathsf{vec}_0 = \left( 1, b_1-r+1, b_2+n-1+\frac{1}{2}(r-1-2b_1)(r-2) \right)^\intercal$ to \out\;\label{first3coeffs}

$\displaystyle \fixed \gets \left\{ \bm{\gamma} \geqslant \bm{0}, \bar{\bm{\gamma}}(1) = 1, \bar{\bm{\gamma}}(2) = \mathsf{vec}_0(2), \bar{\bm{\gamma}}(3) = \mathsf{vec}_0(3) \right\}$\;

\For{$i = 3$ \KwTo $r-1$}
        {
            \For{$j = 1$ \KwTo $\# \out$}
            {
                $\vec \gets$ the $j$th vector in \out, and $\constraints \gets \fixed$\;

                \lFor{$k=4$ \KwTo $i$}
                {append $\bar{\bm{\gamma}}(k) = \vec(k)$ to \constraints}
                
                $\minlp \gets$ the minimum value of $\bar{\bm{\gamma}}(i+1)$ subject to \constraints\;\label{LPmin}
                
                $\maxlp \gets$ the maximum value of $\bar{\bm{\gamma}}(i+1)$ subject to \constraints\;\label{LPmax}

                $\displaystyle \start \gets 2^i \left\lceil \frac{\mathsf{Round}(\minlp)}{2^i} \right\rceil$, and $\displaystyle \en \gets 2^i \left\lfloor \frac{\mathsf{Round}(\maxlp)}{2^i} \right\rfloor$\;

                \If{$i=r-1$\label{ilast}}
                {
                    \lIf{the region bounded by \constraints and $\bar{\bm{\gamma}}(r)=\start$ is infeasible}{$\start \gets \start+2^i$}
                    \lIf{the region bounded by \constraints and $\bar{\bm{\gamma}}(r)=\en$ is infeasible}{$\en \gets \en-2^i$}\label{infeasible}
                }
                
                \For{$l=\start$ \KwTo $\en$ \KwBy $2^i$}
                {Append the vector $(\vec(1), \vec(2), \dots, \vec(i), l)^\intercal$ to \temp\;}
            }
            $\out \gets \temp$, and $\temp \gets \varnothing$\;
        }

Convert the vectors in \out to univariate polynomials in $x$\;

Update \out by substituting $x$ with $x+1$ for each polynomial in \out\;

Multiply each polynomial in \out by $\gcd(p(x), p^\prime(x))$\;

\Return{\out}\;

\end{algorithm}

The inputs of \hyperref[alg:interlacingLPevenodd]{$\mathsf{InterlacingPartial}$} are Seidel trace polynomial $p(x)$ of even degree $n \geqslant 6$ where $3 \leqslant \deg \Min(p, x) \leqslant n$, and an integer $\delta$ where $\max(4, \deg \Min(p, x)-1) \leqslant \delta \leqslant n-1$.
The output is the set $\mathcal{I}(p(x), \delta)$.
Let $r$ be the degree of $\Min(p, x)$ so $3 \leqslant r \leqslant n$.
Define $\Quo(p, x) \coloneqq \gcd(p(x), p^\prime(x))$ and let $q(x) \in \mathcal{I}(p(x), \delta)$.
By Proposition~\ref{prop:interlacingchifactor}, we have that $q(x) = \Quo(p, x) \cdot f(x)$ for some monic real-rooted integer polynomial $f(x)$ that interlaces $\Min(p, x)$.
\hyperref[alg:interlacingLPoddeven]{$\mathsf{InterlacingEven}$} and \hyperref[alg:interlacingLPevenodd]{$\mathsf{InterlacingPartial}$} are largely similar, except for the divisibility and number-theoretic conditions that have to be satisfied by $\coeff(q(x-1))$, and consequently, $\coeff(f(x-1))$.
Line~\ref{adaptstart} to Line~\ref{adaptend} of \hyperref[alg:interlacingLPevenodd]{$\mathsf{InterlacingPartial}$} are adapted from Line~\ref{branchout} to Line~\ref{i>3end} of \hyperref[alg:Seidelcharalgodd]{$\mathsf{FeasiblePartial}$}, where we utilise \hyperref[alg:checkmod]{$\mathsf{ModCheck}$}.
Furthermore, in Line~\ref{rlessn} and Line~\ref{qbarfilter} of \hyperref[alg:interlacingLPevenodd]{$\mathsf{InterlacingPartial}$}, if $r \leqslant n-1$ then we remove any candidate $\bar{q}(x) = q(x-1)$ if $\bar{q}(x)$ is not weakly type~2, or if there exists an integer $i$ such that $\max(4, r) \leqslant i \leqslant \delta$ and $\hyperref[alg:checkmod]{\mathsf{ModCheck}}(\bar{q}(x), i)$ returns \texttt{False}.

\begin{algorithm}[htbp]
\capstart
\caption{$\mathsf{InterlacingPartial}(p(x), \delta)$}
\label{alg:interlacingLPevenodd}
\SetKwData{out}{out}
\SetKwData{temp}{temp}
\SetKwData{fixed}{fixed}
\SetKwData{constraints}{constraints}
\SetKwData{vec}{vec}
\SetKwData{minlp}{minLP}
\SetKwData{maxlp}{maxLP}
\SetKwData{pow}{pow}
\SetKwData{start}{start}
\SetKwData{en}{end}
\SetKw{KwBy}{by}
\SetKwFunction{True}{True}
\SetKwInOut{Input}{Input}\SetKwInOut{Output}{Output}
\Input{A Seidel trace polynomial $p(x)$ of even degree $n \geqslant 6$ where $3 \leqslant \deg \Min(p, x) \leqslant n$, and an integer $\delta$ where $\max(4, \deg \Min(p, x)-1) \leqslant \delta \leqslant n-1$.}
\Output{The set $\mathcal{I}(p(x), \delta)$ of all $\delta$-partial Seidel interlacing polynomials of $p(x)$.}

$r \gets \deg \Min(p, x)$, and $\Quo(p, x) \gets \gcd(p(x), p^\prime(x))$\;

Let $\displaystyle \Min(p, x) = \sum_{i=0}^r b_i x^{r-i}$ for some integers $b_0, b_1, \dots, b_r$\;

Initialise each of \temp and \out as an empty array\;

Let $B$ be a square matrix of order $r$ such that for all $i, j \in \{1, 2, \dots, r\}$, we have that $B_{i, j}$ is equal to a numerical approximation of the $i$th entry of $\coeff(\Min_j(p, x-1))$ with a suitably high digit precision\;

Initialise $\bm{\gamma}$ as a vector of variables $\gamma_1, \gamma_2, \dots, \gamma_r$, and $\bar{\bm{\gamma}} \gets B \bm{\gamma}$\;

Append the vector $\mathsf{vec}_0 = \left( 1, b_1-r+1, b_2+n-1+\frac{1}{2}(r-1-2b_1)(r-2) \right)^\intercal$ to \out\;

$\displaystyle \fixed \gets \left\{ \bm{\gamma} \geqslant \bm{0}, \bar{\bm{\gamma}}(1) = 1, \bar{\bm{\gamma}}(2) = \mathsf{vec}_0(2), \bar{\bm{\gamma}}(3) = \mathsf{vec}_0(3) \right\}$\;

\For{$i = 3$ \KwTo $r-1$}
        {
            \For{$j = 1$ \KwTo $\# \out$}
            {
                $\vec \gets$ the $j$th vector in \out, and $\constraints \gets \fixed$\;

                \lFor{$k=4$ \KwTo $i$}
                {append $\bar{\bm{\gamma}}(k) = \vec(k)$ to \constraints}
                
                $\minlp \gets$ the minimum value of $\bar{\bm{\gamma}}(i+1)$ subject to \constraints\;
                
                $\maxlp \gets$ the maximum value of $\bar{\bm{\gamma}}(i+1)$ subject to \constraints\;

                \eIf{$i=3$\label{adaptstart}}
                {
                    $\pow \gets 2$, and $\displaystyle \start \gets 4 \left\lceil \frac{\mathsf{Round}(\minlp)}{4} \right\rceil$\;
                }
                {
                    $\pow \gets i$, $\displaystyle C_1 \gets 2^{\pow-1} \left\lceil \frac{\mathsf{Round}(\minlp)}{2^{\pow-1}} \right\rceil$, and $C_2 \gets C_1+2^{\pow-1}$\;
                    $\bar{P}_1(x) \gets C_1 \cdot x^{r-1-i}$, and $\bar{P}_2(x) \gets C_2 \cdot x^{r-1-i}$\;
                    \For{$m=0$ \KwTo $i-1$}{$\bar{P}_1(x) \gets \bar{P}_1(x) + \vec(m+1) \cdot x^{r-1-m}$, and $\bar{P}_2(x) \gets \bar{P}_2(x) + \vec(m+1) \cdot x^{r-1-m}$\;}
                    \uIf{$\hyperref[alg:checkmod]{\mathsf{ModCheck}}$\(\left( \Quo(p, x-1) \cdot \bar{P}_1(x), i \right) = \True\)}{$\start \gets C_1$\;}
                    \uElseIf{$\hyperref[alg:checkmod]{\mathsf{ModCheck}}$\(\left( \Quo(p, x-1) \cdot \bar{P}_2(x), i \right) = \True\)}{$\start \gets C_2$\;}
                    \Else{continue\;\label{adaptend}}
                }
                
                $\displaystyle \en \gets \start + 2^{\pow} \left\lfloor \frac{\mathsf{Round}(\maxlp)-\start}{2^{\pow}} \right\rfloor$\;

                \If{$i=r-1$}
                {
                    \lIf{the region bounded by \constraints and $\bar{\bm{\gamma}}(r)=\start$ is infeasible}{$\start \gets \start+2^{\pow}$}
                    \lIf{the region bounded by \constraints and $\bar{\bm{\gamma}}(r)=\en$ is infeasible}{$\en \gets \en-2^{\pow}$}
                }
                
                \For{$l=\start$ \KwTo $\en$ \KwBy $2^{\pow}$}
                {Append the vector $(\vec(1), \vec(2), \dots, \vec(i), l)^\intercal$ to \temp\;}
            }
            $\out \gets \temp$, and $\temp \gets \varnothing$\;
        }

Convert the vectors in \out to univariate polynomials in $x$\;

Multiply each polynomial in \out by $\Quo(p, x-1)$;

\If{$r \leqslant n-1$\label{rlessn}}{Remove any polynomial $\bar{q}(x)$ in \out if $\bar{q}(x)$ is not weakly type~2, or if there exists an integer $i$ such that $\max(4, r) \leqslant i \leqslant \delta$ and $\hyperref[alg:checkmod]{\mathsf{ModCheck}}(\bar{q}(x), i) = \False$\;}\label{qbarfilter}

Update \out by substituting $x$ with $x+1$ for each remaining polynomial in \out\;

\Return{\out}\;

\end{algorithm}

Finally, Example~\ref{ex:interlacingseidelchi} below illustrates a method to show whether a given integer polynomial $p(x)$ is not Seidel-realisable by utilising the set $\mathcal{I}(p(x))$ and \cite[Theorem 2.3]{GSY21}.
This strategy is central to and is a recurring theme in \cite{GS24, GSY21, GSY23}.

\begin{ex} \label{ex:interlacingseidelchi}
    Let $p(x) = (x+1)^2(x-1)(x^3-x^2-13x+5) \in \mathcal{F}_6$.
    By running \hyperref[alg:numrealrooted]{$\mathsf{AllRealRooted}$} and checking interlacing, we find that the cardinality of $\mathcal{J}(p(x))$ is 67.
    By running \hyperref[alg:interlacingLPevenodd]{$\mathsf{InterlacingPartial}$}, we find that the set $\mathcal{I}(p(x), 4)$ consists of three polynomials: 
    \begin{align*}
        q_1(x) &= (x+1)^2 (x^3-2x^2-7x+4), \\
        q_2(x) &= (x+1)^2(x-1)(x^2-x-8), \\
        q_3(x) &= (x+3)(x+1)x(x-1)(x-3).
    \end{align*}
    Furthermore, we have $\mathcal{I}(p(x)) = \mathcal{I}(p(x), 5) = \{ q_2(x) \}$.
    By Remark~\ref{rem:deg5n6}, we have that $q_2(x) \in \mathcal{R}_5$.
    However, since $6q_2(x) \neq p^{\prime}(x)$, by \cite[Theorem 2.3]{GSY21}, we conclude that $p(x) \notin \mathcal{R}_6$.
\end{ex}

\bibliographystyle{amsplainedit}
\bibliography{biblio}

\end{document}